\documentclass[a4paper, 10pt, twoside, notitlepage]{amsart}

\usepackage[utf8]{inputenc}
\usepackage{xcolor}
\usepackage{amsmath} 
\usepackage{amssymb} 
\usepackage{amsthm}
\usepackage{geometry}
\usepackage{graphicx}
\usepackage{mathtools}
\usepackage{esint}
\usepackage{enumitem}
\usepackage{comment}
\usepackage[colorlinks=true,linkcolor=blue]{hyperref}

\theoremstyle{plain}
\newtheorem{thm}{Theorem}
\newtheorem{prop}{Proposition}[section]
\newtheorem{lem}[prop]{Lemma}
\newtheorem{cor}[prop]{Corollary}

\newtheorem{rmk}[prop]{Remark}

\newcommand {\R} {\mathbb{R}} 
 \newcommand {\N} {\mathbb{N}}
 
\newcommand {\p} {\partial}

\newcommand {\D} {\Delta}

\newcommand {\supp} {\text{supp}}

\DeclareMathOperator {\Imm} {Im}

\DeclareMathOperator{\F} {\mathcal{F}}

\allowdisplaybreaks

\begin{document}

\title[Anisotropic Fractional Calder\'on]{Revisiting the Anisotropic Fractional Calder\'on Problem Using the Caffarelli-Silvestre Extension} 

\author[Angkana R\"uland]{Angkana R\"uland}
\address{Institute for Applied Mathematics and Hausdorff Center for Mathematics, University of Bonn, Endenicher Allee 60, 53115 Bonn, Germany}
\email{rueland@uni-bonn.de}

%\subjclass{53C25, 53C21, 58F17, 35J15}

%\date{\today}

%\maketitle

\begin{abstract}
We revisit the source-to-solution anisotropic fractional Calder\'on problem introduced and analyzed in \cite{FGKU21} and \cite{F21}. Using the variable coefficient Caffarelli-Silvestre-type extension interpretation of the fractional Laplacian from \cite{ST10}, we provide an alternative argument for the recovery of the heat and wave kernels from \cite{FGKU21}. This shows that in the setting of the source-to-solution anisotropic fractional Calder\'on problem the heat and the (degenerate) elliptic extension approach give rise to equivalent perspectives and that each kernel can be recovered from the other. Moreover, we also discuss the Dirichlet-to-Neumann anisotropic source-to-solution problem and provide a direct link between the Dirichlet Poisson kernel and the wave kernel. This illustrates that it is also possible to argue completely on the level of the Poisson kernel, bypassing the recovery of the heat kernel as an additional auxiliary step. Last but not least, as in \cite{CGRU23}, we relate the local and nonlocal source-to-solution Calder\'on problems.
\end{abstract}

\maketitle

\section{Introduction} \label{sec_introduction}
The objective of this article is three-fold: Firstly, we compare two interpretations of the (anisotropic) fractional Calder\'on problem, the heat interpetation on the one hand, and the variable coefficient Caffarelli-Silvestre-type extension representation from \cite{ST10} on the other hand. We remark that, on a formal level, both of these interpretations involve an additional coordinate (in time/space, respectively). We seek to prove that on the level of the inverse problem both of these interpretations carry the same information. Building on this, we revisit the seminal uniqueness result of \cite{FGKU21} for the anisotropic fractional Calder\'on problem with source-to-solution data on a smooth, closed, connected manifold and provide a variant of the proof of uniqueness by means of the variable coefficient Caffarelli-Silvestre-type extension representation. 
We believe that the comparison of these perspectives may be of interest in providing different approaches towards the (anisotropic) fractional Calder\'on problem which could be of use for a further analysis of its properties.

Secondly, we discuss the setting of the Dirichlet-to-Neumann source-to-solution map for which we deduce an analogous uniqueness result and for which we show that it is possible to bypass the heat kernel and directly relate the wave and Poisson kernels. 

Third and finally, we relate the local and nonlocal Calder\'on problems with source-to-solution data in parallel to the setting from \cite{CGRU23} which involves indirect measurements encoded in boundary (instead of source-to-solution) data.

\subsection{The setting}
Let us begin by recalling the setting of the (anisotropic) fractional Calder\'on problem with source-to-solution data from \cite{FGKU21}.
Let $(M,g)$ be a closed, connected, smooth manifold of dimension $n\geq 2$, let $s\in (0,1)$ and consider
\begin{align}
\label{eq:source}
(-\D_g)^s u = f \mbox{ in } M,
\end{align}
with $u\in H^s(M)$ and $f\in H^{-s}(M)$. The operator $(-\D_g)^s u$ can be defined spectrally, i.e., for $\{\phi_k\}_{k\in \N} \subset L^2(M)$ an $L^2$-normalized eigenfunction basis for the Laplace-Betrami operator $-\D_g$ on $(M,g)$ with (not necessarily distinct) eigenvalues $\{\lambda_k\}_{k\in \N}$ and for $u\in H^s(M)$, we can define
\begin{align}
\label{eq:spectral_Lapl}
(-\D_g)^s u(x):= \sum\limits_{k\in \N} \lambda_k^{s} (u,\phi_k)_{L^2(M)} \phi_k(x).
\end{align}
We recall that for $(M,g)$ a smooth, closed, connected manifold, it holds that the zero eigenvalue occurs with multiplicity one, i.e., assuming that the (not necessarily distinct) eigenvalues are ordered by increasing size, we have that $0=\lambda_0< \lambda_1 \leq \lambda_2 \leq ...$ and (up to normalization) the zeroth eigenfunction satisfies $\phi_0(x) \equiv 1$.

In this article, in many places we also view the anisotropic fractional Laplacian $(-\D_g)^s u$ in terms of its variable coefficient Caffarelli-Silvestre-type extension characterization given in \cite{ST10}. Indeed, we rely on the  general extension theory which was developed in a series of works for the settings of constant coefficients, variable coefficients or of conformally invariant, geometric elliptic operators, respectively \cite{CS07,ST10, CS16,CG11,BGS15}. In particular, these results assert that for the closed, connected, smooth manifold $(M,g)$ and on $C^{\infty}(M)$ the spectral fractional Laplacian can be realized through its variable coefficient Caffarelli-Silvestre-type extension: More precisely, assuming that $u\in C^{\infty}(M)$ and 
$\tilde{u}\in \dot{H}^{1}(M\times \R_+,x_{n+1}^{1-2s})$ is the solution to
\begin{align}
\label{eq:CS1}
\begin{split}
(x_{n+1}^{1-2s} \D_g + \p_{n+1} x_{n+1}^{1-2s} \p_{n+1}) \tilde{u} & = 0 \mbox{ in } M \times \R_+,\\
\tilde{u} &= u \mbox{ on } M \times \{0\},
\end{split}
\end{align}
then, by the variable coefficient extension results from \cite{ST10}, in a weak sense, it holds that for some constant $c_s\neq 0$
\begin{align}
\label{eq:CS2}
(-\D_g)^s u(x):= c_{s} \lim\limits_{x_{n+1} \rightarrow 0} x_{n+1}^{1-2s} \p_{n+1} \tilde{u}(x,x_{n+1}), \ x\in M.
\end{align}
By a density argument, this remains true for $u\in H^s(M)$.
In what follows, in many parts of the article, we will work with the characterization given in \eqref{eq:CS1}, \eqref{eq:CS2}. For notational simplicity, in what follows, we will often simply refer to this extension as a (variable coefficient) Caffarelli-Silvestre-type extension.

As in \cite{FGKU21} we consider the source-to-solution map $L_{s,O}$ for some open set $O \subset M$:
\begin{align}
\label{eq:sts1}
L_{s,O}: \tilde{H}^{-s}(O) \rightarrow H^{s}(O), \ f \mapsto u|_{O},
\end{align}
where $f$ is such that $(f,1)_{L^2(M)}=0$ and $u$ is the solution to \eqref{eq:source} with $(u,1)_{L^2(M)}=0$. We refer to Section \ref{sec:not} for the definition of the relevant function spaces. As in \cite{FGKU21} this can be interpreted spectrally by noting that, with the eigenbasis and eigenvalues introduced above and since $(f,1)_{L^2(M)}=0$, the unique solution to \eqref{eq:source} with vanishing mean value, $(u,1)_{L^2(M)}=0$, can be expressed as 
\begin{align}
\label{eq:mapsts}
u(x) = \sum\limits_{k=1}^{\infty} \lambda_k^{-s} (f,\phi_k)_{L^2(M)} \phi_k(x).
\end{align}
By virtue of the result from \cite{ST10} this can also be formulated in terms of a variable coefficient Caffarelli-Silvestre-type extension operator: In order to state this result, as in \cite{ST10}, we introduce, with slight abuse of notation, the (Neumann) Poisson kernel $P_{y}(x,z)$ for the variable coefficient Caffarelli-Silvestre-type extension (see also \eqref{eq:Poisson} below). It maps $f\in C^{\infty}(M)$ with $(f,1)_{L^2(M)}=0$ to the unique solution $\tilde{u} \in \dot{H}^{1}(M\times \R_+, x_{n+1}^{1-2s})$ with $(\tilde{u}(\cdot, x_{n+1}),1)_{L^2(M)}=0$ for all $x_{n+1}\in \R_+$ and
\begin{align}
\label{eq:Neu}
\begin{split}
(x_{n+1}^{1-2s} \D_g + \p_{n+1} x_{n+1}^{1-2s} \p_{n+1}) \tilde{u} & = 0 \mbox{ in } M \times \R_+,\\
\lim\limits_{x_{n+1}\rightarrow 0} x_{n+1}^{1-2s}\p_{n+1} \tilde{u} &= f \mbox{ on } M \times \{0\}.
\end{split}
\end{align}
By the results of \cite{CS07,ST10} the map $C^{\infty}(M) \ni f \mapsto \tilde{u}(x,0)|_{O}$ is equal to the spectral interpretation of the source-to-solution map $L_{s,O}(f)$ given in \eqref{eq:sts1} and \eqref{eq:mapsts} (see also Appendix \ref{app:exp} for similar arguments).

In this set-up, in \cite{FGKU21}, the following seminal result on the uniqueness up to natural gauge of an anisotropic metric is proved:

\begin{thm}[\cite{FGKU21}]
\label{thm:FGKU21}
Let $s\in (0,1)$, assume that $(M_1,g_1)$, $(M_2, g_2)$ are smooth, closed, connected manifolds of dimension $n\geq 2$. Let $(O_1,g_1)\subset (M_1,g_1)$ and $(O_2,g_2)\subset (M_2,g_2)$ be open sets such that $(O_1,g_1)=(O_2,g_2)=:(O,g)$. Let $L_{s,O}^{1}, L_{s,O}^2$ denote the source-to-solution operators associated with the metrics $g_1,g_2$ as in \eqref{eq:sts1} and assume that $L_{s,O}^{1} = L_{s,O}^2$. Then there exists a diffeomorphism $\Phi$ such that $\Phi^* g_2 = g_1$.
\end{thm}

As a first main result in this article it is our objective to reprove the first step in the uniqueness result from \cite{FGKU21} using the variable coefficient Caffarelli-Silvestre-type extension from \cite{ST10}. Here we first recover the Neumann Poisson kernel of the variable coefficient Caffarelli-Silvestre-type extension (see equation \eqref{eq:Poisson} for its definition) and, then, using the relation between the two kernels, recover the heat kernel. We argue that the Neumann Poisson and heat kernel carry equivalent information in this context (c.f. Proposition \ref{prop:equiv} below). 

As a second main result, we remark that an analogous uniqueness result also holds in the case in which one considers a source-to-solution operator with prescribed Dirichlet and measured Neumann data. More precisely, for $c_s \neq 0$ chosen as above, consider the source-to-solution operator
\begin{align}
\label{eq:sts2}
\tilde{L}_{s,O}: \tilde{H}^{s}(O) \rightarrow H^{-s}(O), \ f \mapsto c_s \lim\limits_{x_{n+1}\rightarrow 0} x_{n+1}^{1-2s} \p_{x_{n+1}} \tilde{v}(x,x_{n+1})|_{O},
\end{align}
where $\tilde{v} \in \dot{H}^1(M \times \R_+, x_{n+1}^{1-2s})$ is the weak solution to
\begin{align}
\label{eq:Dir}
\begin{split}
(x_{n+1}^{1-2s} \D_g + \p_{n+1} x_{n+1}^{1-2s} \p_{n+1}) \tilde{v} & = 0 \mbox{ in } M \times \R_+,\\
\tilde{v} &= f \mbox{ on } M \times \{0\}.
\end{split}
\end{align}
We remark that, as outlined above, by the results from \cite{CS07,ST10, CS16} the map $\tilde{L}_{s,O}$ can also be rewritten as the map
\begin{align*}
\tilde{L}_{s,O}: \tilde{H}^{s}(O) \rightarrow H^{-s}(O), \ f \mapsto (-\D_g)^s f|_{O},
\end{align*}
where $(-\D_g)^s f$ denotes the spectral Laplacian on $(M,g)$ as defined in \eqref{eq:spectral_Lapl}. By virtue of the results from \cite{ST10}, similarly as in the Neumann case discussed above, this map can be encoded in terms of a (Dirichlet) Poisson operator $P^D_y(x,z)$, defined in \eqref{eq:PoissonD} below.

In this context, it is now possible to prove the analogue of the uniqueness result from \cite{FGKU21} in the setting of Dirichlet-to-Neumann measurement data:

\begin{prop}
\label{prop:DtN}
Let $s\in (0,1)$, assume that $(M_1,g_1)$, $(M_2, g_2)$ are smooth, closed, connected manifolds of dimension $n\geq 2$. Let $(O_1,g_1)\subset (M_1,g_1)$ and $(O_2,g_2)\subset (M_2,g_2)$ be open sets such that $(O_1,g_1)=(O_2,g_2)=:(O,g)$. Let $\tilde{L}_{s,O}^{j}$, $j\in \{1,2\}$, denote the associated source-to-solution operators as in \eqref{eq:sts2} and assume that $\tilde{L}_{s,O}^{1} = \tilde{L}_{s,O}^2$. Then there exists a diffeomorphism $\Phi$ such that $\Phi^* g_2 = g_1$.
\end{prop}

We emphasize that, on a formal level, this can also be deduced as a consequence of the uniqueness result from \cite{FGKU21} by using the associated, generalized Cauchy-Riemann equations for the variable coefficient extension problem. We refer to Remark \ref{rmk:CR} in Section \ref{sec:DtN} below for a more detailed outline of this.

\begin{rmk}
\label{rmk:Dir_nonlocal}
We highlight that the problem formulated in Proposition \ref{prop:DtN} is only solvable for non-integer $s\in (0,1)$ and that it is immediate that for $s=1$ it \emph{cannot} give rise to uniqueness up to natural gauges. Indeed, for $s=1$, the accessible data would correspond to the knowledge of the map $f\mapsto (-\D_g)f|_O$ for all $f\in C_c^{\infty}(O)$. By locality of the fractional Laplacian, this cannot determine the metric $g$ uniquely. The uniqueness result of Proposition \ref{prop:DtN} therefore is a \emph{purely nonlocal} phenomenon and does \emph{not} have a direct local analogue.
\end{rmk}

Let us briefly comment on the ideas of the proofs of Theorem \ref{thm:FGKU21} and Proposition \ref{prop:DtN} by means of the Neumann and Dirichlet Poisson kernels and the variable coefficient Caffarelli-Silvestre-type extension. The key step in our argument is the reconstruction of the Poisson kernel (in its Neumann and Dirichlet variant, respectively) from the data $L_{s,O}$ and $\tilde{L}_{s,O}$, respectively. To this end, we use that the Poisson kernels can be expanded in a power series in the new normal variable which is introduced by the variable coefficient Caffarelli-Silvestre-type extension, provided the data are measured on disjoint domains. This allows us to recover the respective Poisson kernels. In a second step, we then reduce the knowledge of the respective Poisson kernel to the knowledge of the heat kernel. To this end, we exploit that, by the results and arguments of \cite{ST10} and \cite{CS16}, the Neumann Poisson kernel $P_{y}(x,z)$ can be related to the heat kernel $K_t(x,z)$ on $(M,g)$
\begin{align}
\label{eq:Poisson}
P_{y}(x, z):= \tilde{c}_s \int\limits_{0}^{\infty} K_{t}(x,z) e^{-\frac{y^2}{4t}} \frac{dt}{ t^{1-s}}, \ x,z\in M, \ \tilde{c}_s \in \R \setminus \{0\}.
\end{align}

With this notation fixed, we have the following equivalence of measurement result.

\begin{prop}[Equivalence of the heat and Poisson kernel reconstruction]
\label{prop:equiv}
Let $(M,g)$ be a closed, connected, smooth manifold.
Let $K_{t}(x,z)$ and $P_{y}(x,z)$ be as above. Assume that $x,z\in M$ with $x\neq z$. Then, $K_{t}(x,z)$ and $P_{y}(x,z)$ carry the same information. More precisely, if $K_{t}(x,z)$ is known for all $t\in (0,\infty)$ it is possible to recover $P_{y}(x,z)$ for all $y\in (0,\infty)$ and vice versa.
\end{prop}

\begin{rmk}
A similar equivalence also holds if one replaces the Neumann Poisson kernel by the Dirichlet Poisson kernel, see Section \ref{sec:DtN} below.
\end{rmk}

Once the equivalence of Proposition \ref{prop:equiv} is obtained, a unique continuation argument yields the equality of the heat data in the whole of the set $O$. At this point it is then possible to complete the proofs of Theorem \ref{thm:FGKU21} and Proposition \ref{prop:DtN} by the same argument as in \cite{FGKU21}: One first applies a Kannai transform to recover the local source-to-solution map for the wave kernel and then invokes the boundary control method. We do not repeat these arguments in this note. However, we remark that in the setting of the Dirichlet realization of Proposition \ref{prop:DtN}, also a more direct argument can be given which avoids the auxiliary step through the heat kernel realization. This is based on the observation that, as for the heat kernel, also the Dirichlet Poisson kernel characterized in \cite{ST10}
\begin{align}
\label{eq:PoissonD}
P^{D}_{y}(x,z):= c_s y^{2s} \int\limits_{0}^{\infty} K_t(x,z) e^{-\frac{y^2}{4t}} \frac{dt}{ t^{1+s}}, \ x,z \in M, \ c_s \in \R \setminus \{0\},
\end{align}
can be directly related to the wave kernel through a ``Kannai-type transform'':

\begin{prop}
\label{prop:Kannai}
Let $(M,g)$ be a smooth, closed, connected manifold. Let $s\in (0,1)$, let $y>0$ and let $x,z \in M$. Let further 
$P_{y}^D(x,z)$ denote the Dirichlet Poisson kernel \eqref{eq:PoissonD}. Denote for $t>0$ and $x,z \in M$ the wave kernel on $(M,g)$  by $K^{w,g}_{t}(x,z):=\frac{\sin(t\sqrt{(-\D_g)})}{\sqrt{(-\D_g)}}(x,z)$, i.e., suppose that for any $f\in C^{\infty}( M)$ the unique solution $u \in C^{\infty}([0,\infty),M)$ of
\begin{align*}
(\p_t^2 - \Delta_g) u& = 0 \mbox{ on } (0,\infty)\times M,\\
u&= 0 \mbox{ on } \{0\}\times M,\\
\p_t u & = f \mbox{ on } \{0\}\times M,
\end{align*}
is expressed as $u(t,x)= \int\limits_{M} \frac{\sin(t\sqrt{(-\D_g)})}{\sqrt{(-\D_g)}}(x,z) f(z) dz$ for $t>0, x \in M$.

Then there exists a constant $\bar{c}_{s,D} \neq 0$ such that for $y>0$, $x\in M$ and for any $f\in C^{\infty}(M)$
\begin{align*}
\int\limits_{M} P_y^D(x,z) f(z) dz & = \bar{c}_{s,D} y^{2s} \int\limits_M \int\limits_{0}^{\infty} (\tau + y^2)^{-\frac{3}{2}-s} \frac{\sin(\sqrt{\tau} \sqrt{(-\D_g)})}{\sqrt{(-\D_g)}}(x,z) d \tau f(z) dz.
\end{align*}
\end{prop}

\begin{rmk}
\label{rmk:Duhamel}
In order to relate the Dirichlet Poisson and wave source-to-solution operators $\tilde{L}_{s,O}^g$ from \eqref{eq:sts2} and $L^{\text{wave}}_{g,O}$ from \eqref{eq:wavests} in what follows below, we note that, by Duhamel's principle, for any $F\in C^{\infty}_c((0,\infty) \times M)$ the unique solution $u \in C^{\infty}([0,\infty),M)$ of
\begin{align*}
(\p_t^2 - \Delta_g) u& = F \mbox{ on } (0,\infty)\times M,\\
u, \p_t u & = 0 \mbox{ on } \{0\}\times M,
\end{align*}
is expressed as $u(t,x)= \int\limits_{0}^t \int\limits_{M} \frac{\sin((t-s)\sqrt{(-\D_g)})}{\sqrt{(-\D_g)}}(x,z) F(s,z) dz ds$.

We note that these expressions can be interpreted spectrally.
\end{rmk}

As in the heat case, the transformation from Proposition \ref{prop:Kannai} can be inverted (see Proposition \ref{prop:Kannai_inversion} below). Hence, it is possible to access the source-to-solution map for the wave equation, once the source-to-solution map for the Dirichlet Poisson kernel is known. This then allows to directly invoke the results from the boundary control method to provide an alternative proof of Proposition \ref{prop:DtN} without resorting to the auxiliary step of the heat kernel. We refer to Section \ref{sec:Kannai} for the details of this Kannai-type transform.

Finally, as a last main result of this article, we relate the local and nonlocal Calder\'on problems on $(M,g)$ with source-to-solution data (in parallel to the boundary data setting in \cite{CGRU23, GU21}). As in \cite{CGRU23, GU21} an application of this is the recovery of nonlocal uniqueness results from local uniqueness results. The main step towards this consists of a density result in passing from the nonlocal data to the local data, which is of interest in its own right.
We only formulate this result at this point and refer to Section \ref{sec:relation} for the further discussion of the derivation of the nonlocal uniqueness results from their local counterparts.

\begin{thm}
\label{thm:density1}
Let $s\in (0,1)$.
Let $(M,g)$ be a smooth, closed, connected manifold of dimension $n\geq 2$. Let $(O,g) \subset (M,g)$ be an open set.
Let $L_{1,O}$ denote the source-to-solution map
\begin{align*}
\tilde{H}^{-1}(O) \ni f \mapsto v_f|_{O} \in H^1(O), 
\end{align*}
where $v_f $ with $(v_f, 1)_{L^2(M)}=0$ is a solution to $(-\D_g) v_f = f$ in $(M,g)$ with $(f,1)_{L^2(M)}=0$. Further, let $L_{s,O}$ denote the source-to-solution map from \eqref{eq:sts1}. Then, the operator $L_{s,O}$ determines $L_{1,O}$. 
More precisely, we have the following results:
\begin{itemize}
\item[(i)] 
Let $f\in C_c^{\infty}(O)$ with $(f,1)_{L^2(M)}=0$ and let $\tilde{u}_f \in \dot{H}^1(M \times \R_+, x_{n+1}^{1-2s})$ with the property that $(\tilde{u}(\cdot, x_{n+1}),1)_{L^2(M)}=0$ for all $x_{n+1} \in \R_+$ denote the variable coefficient Caffarelli-Silvestre-type extension of $u_f$ (as in \eqref{eq:CS1}) and assume that $u_f$ with $(u_f, 1)_{L^2(M)}=0$ solves $(-\D_g)^s u_f = f$ in $M$. Then, for any $f\in C_c^{\infty}(O)$, the function
\begin{align*}
\int\limits_{0}^{\infty} t^{1-2s} \tilde{u}_f(x,t) dt \big|_{O}
\end{align*}
can be reconstructed from the pair $(f,L_{s,O}f)$ and the metric $g|_O$.
\item[(ii)]
With the notation as in (i), the following density result holds:
\begin{align*}
&\overline{\left\{ \left( f, \int\limits_{0}^{\infty} t^{1-2s} \tilde{u}_f(x,t)dt \big|_{O} \right) \in \tilde{H}^{-1}(O)\times H^1(O): \  f\in C_c^{\infty}(O) \mbox{ with } (f,1)_{L^2(M)}=0 \right\}}^{H^{-1}(O)\times H^1(O)} \\
&= \{(f,L_{1,O}f): \ f\in \tilde{H}^{-1}(O) \mbox{ with } (f,1)_{L^2(M)}=0\}.
\end{align*}
\end{itemize}
\end{thm}

\begin{rmk}
We remark that a dimension count shows that, contrary to the setting from \cite{CGRU23}, the two data sets given by the source-to-solution measurements of the local and nonlocal operators, now formally carry the same amount of information. This is reflected in Lemma \ref{lem:relations} (ii), in which the relation is reversed under the assumption that the metric is globally known. It remains an open problem to reverse this in the setting of the inverse problem in which the metric is only known in $O\subset M$.
\end{rmk}

\subsection{Relation to the literature}

Since its introduction in the seminal article \cite{GSU16}, in which partial data uniqueness results had been deduced, the fractional Calder\'on problem has been studied intensively as a prototypical, nonlocal, nonlinear elliptic inverse problem in which striking novel features arise compared to the classical, local Calder\'on problem. As examples for this, we mention \cite{RS20} for partial data uniqueness and stability results at low regularity, \cite{GRSU20, R21} for single measurement reconstruction and stability results, \cite{HL19,HL20} for reconstruction by monotonicity methods, \cite{C20b} for the introduction of a conductivity formulation and a Liouville reduction, \cite{GLX17} for the recovery of potentials in the presence of (known) anisotropic background metrics, \cite{LL22,KLZ22} for uniqueness in nonlinear models, and \cite{RS18} for instability results. Further, the methods developed in these articles are rather robust and have been extended to various related models such as \cite{L23, CdHS22} for fractional Lam\'e type problems, \cite{RS20, LLR20} for different fractional heat problems, \cite{C20,CLR20,BGU21,L21,L20} for operators involving lower order local and nonlocal contributions, \cite{GFR19, CMRU20, LLL23} for higher order operators, \cite{CLL19} for single measurement obstacle and potential reconstruction, \cite{GFR20, CGFR22, LZ23} for more general combinations of local and nonlocal operators, and \cite{RZ22,RZ22b,CRTZ22} for work on the fractional conductivity problem. We emphasize that this list of references is not exhaustive and that, for instance, many further results can found in the references to these articles and in the survey articles \cite{S17,R18}.

While most of these results mainly focused on the recovery of lower order contributions of the operators, a recent line of research initiated by the seminal work of Feizmohammadi \cite{F21} and Ghosh-Feizmohammadi-Krupchyk-Uhlmann \cite{FGKU21} deals with the recovery of the (possibly anisotropic) principle part of the nonlocal operator. Also in this context, the nonlocal results are of striking strength compared to the local setting. For instance, in \cite{FGKU21} it is shown that it is possible to recover an anisotropic conductivity from local source-to-solution data without analyticity assumptions, which is a fundamental open problem in the local context, see \cite{U09} for a survey and references. These results have subsequently been extended to include fractional anisotropic, geometric operators in \cite{QU22, C22} and also to other nonlocal generalizations of anisotropic fractional operators in \cite{Co22}. In a related line of research, in \cite{GU21, CGRU23} a connection between the (anisotropic) local and nonlocal Calder\'on problems have been deduced which allow to transfer uniqueness results on the local level to the nonlocal setting. Moreover, obstructions to the inversion of this transformation are discussed. Similar connections between local and nonlocal Calder\'on type problems have also been investigated in the context of the heat equation \cite{LLU22, LLU23}. The present article connects to both of these research strands and provides an additional perspective on the source-to-solution problem based on the variable coefficient Caffarelli-Silvestre-type extension.

\subsection{Outline of the article}

The remainder of the article is structured as follows: After briefly recalling and collecting the relevant notation in Section \ref{sec:not}, in Section \ref{sec:heat_CS} we relate the Neumann Poisson and the heat kernel and present the proof of Proposition \ref{prop:equiv}. In Section \ref{sec:Poisson_kernel} we then show how the source-to-solution data $L_{s,O}$ allow to recover the Neumann Poisson kernel. Combined with the equivalence of the heat and Neumann Poisson data, this allows us to provide a variable coefficient Caffarelli-Silvestre-type extension based argument for the first step in the proof of Theorem \ref{thm:FGKU21}. Moreover, we also show that an analogous strategy can be implemented in the case that Dirichlet-to-Neumann source-to-solution data are available and prove the statement of Proposition \ref{prop:DtN}. Section \ref{sec:Kannai} is dedicated to a Kannai-type transform between the Dirichlet Poisson and the wave kernel as stated in Proposition \ref{prop:Kannai}. This allows to bypass the heat kernel in deducing uniqueness results as in \cite{FGKU21}. In Section \ref{sec:relation} we then relate the local and nonlocal source-to-solution anisotropic Calder\'on problems and deduce the result of Theorem \ref{thm:density1}. Finally, in the appendix, we relate the Dirichlet and Neumann Poisson kernels from \eqref{eq:Dir}, \eqref{eq:Neu} with their characterizations in terms of eigenfunctions. 

\subsection{Notation}
\label{sec:not}

In this section, for the convenience of the reader, we fix and collect some of the notation used in the following sections. 

\subsubsection{Function spaces}

For the nonlocal operators discussed in this article, we will work with $L^2$-based fractional Sobolev spaces. Here we will mainly consider the Sobolev spaces with (fractional) exponent $r\in (0,1]$. These will always be considered on $(M,g)$ which here denotes a smooth, closed, connected manifold of dimension $n\geq 2$. Moreover, we denote by $\{\phi_k\}_{k\in\N}$, $\{\lambda_k\}_{k\in \N}$ an $L^2$-normalized eigenbasis and eigenvalues of the Laplace-Beltrami operator $(-\D_g)$ on $(M,g)$. 
\begin{itemize}
\item For $r\in \R$ we set $H^r(M):=\{u:M \rightarrow \R: \ \sum\limits_{j=0}^{\infty} (1+\lambda_k)^r|(u,\phi_k)_{L^2(M)}|^2 <\infty \}$. For $r\in \N$ this can also be expressed in terms of the Levi-Civita connection on $(M,g)$.
\item For $r\in \R$ and for some open set $O \subset M$, we set $\tilde{H}^r(O):=\overline{\{f\in C_{c}^{\infty}(O)\}}^{H^r(M)}$. We recall that if $O \subset M$ is Lipschitz and $r> 0$, then this space is equivalent to several other standard spaces with compact support (see, for instance, the discussion in \cite[Section 2.1]{GSU16}). Further, we recall, that for $O \subset M$ Lipschitz, the following duality relations hold $(H^{r}(O))' = \tilde{H}^{-r}(O)$ and $(\tilde{H}^{r}(O))' = H^{-r}(O)$.
\end{itemize}
We emphasize that in what follows below, we can always, without loss of generality, assume that $O$ is Lipschitz, since else, we consider a smaller, Lipschitz subset of $O \subset M$ and work on this set. Hence, without further mention, we will use the duality properties stated above.

In what follows, we will often consider solutions to \eqref{eq:source} for data $f\in C_c^{\infty}(O)$ with $(f,1)_{L^2(M)}=0$. In order to ensure uniqueness, when referring of solutions of \eqref{eq:source} we will always consider $u_f \in H^{s}(M)$ satisfying \eqref{eq:source} and $(u_f,1)_{L^2(M)}=0$. An analogous convention will be used for solutions to the problem $(-\D_g) v_f = f$ for data $f\in C_c^{\infty}(O)$ with $(f,1)_{L^2(M)}=0$ (and also for the Caffarelli-Silvestre-type formulation \eqref{eq:Neu} of \eqref{eq:source}). 

For the Caffarelli-Silvestre-type extension, for $s\in (0,1)$, for $(M,g)$ a closed, connected, smooth manifold and for $\R_+:=\{y\in \R: y\geq 0\}$ we will mainly work with the space
\begin{align*}
\dot{H}^1(M \times \R_+, x_{n+1}^{1-2s}):= \{ u: M \times \R_+ \rightarrow \R:  \ \|x_{n+1}^{\frac{1-2s}{2}} \nabla u\|_{L^2(M \times \R_+)} <\infty \}.
\end{align*}

Following the notation from \eqref{eq:CS1}, in many instances in the text, when using the notation $u:M \rightarrow \R$ for a function on $M$, we will write $\tilde{u}$ to denote its variable coefficient Caffarelli-Silvestre-type extension $\tilde{u}:M \times \R_+ \rightarrow \R$ with $\tilde{u}\in \dot{H}^1(M \times \R_+, x_{n+1}^{1-2s})$. If $u$ is related to a function $f\in C_c^{\infty}(O)$, for some $O\subset M$ open, through the equations \eqref{eq:source}, \eqref{eq:Neu} or \eqref{eq:Dir}, we also often use a sub- or superscript $f$ to emphasize this dependence as, e.g., in $u_f$, $\tilde{u}_f$. When considering the Neumann extension \eqref{eq:Neu}, we will always additionally require the condition $(\tilde{u}(\cdot, x_{n+1}), 1)_{L^2(M)}=0$ for all $x_{n+1} \in \R_+$ in order to guarantee uniqueness.

\subsubsection{The Neumann and Dirichlet Poisson, the heat and the wave kernels}
With slight abuse of notation, we will refer to the kernels in \eqref{eq:Dir} and \eqref{eq:Neu} as the Dirichlet and Neumann Poisson kernels associated with $(M,g)$ and use the notation $P_y^D(x,z)$, $P_y(x,z)$ with $x,z\in M$, $y\in \R_+$. If we seek to clarify the dependence on the metric $g$, we will also use the notation $P_y^{D,g}(x,z)$, $P_y^g(x,z)$. Similarly, we will denote the heat kernel on $(M,g)$ by $K_t(x,z)$ or $K_t^g(x,z)$ for $x,z\in M$, $t\in \R_+$. Finally, will use the notation
\begin{align*}
K_{\tau}^w(x,z) := K_{\tau}^{w,g}(x,z) :=\frac{\sin(\tau\sqrt{-\D_g})}{\sqrt{-\D_g}}(x,z), \ x,z\in M, \tau\geq 0,
\end{align*}
for the wave kernel from Proposition \ref{prop:Kannai}. The expression on the right hand side can, for instance, be interpreted spectrally. Some spectral characterizations of these kernels can be found in the Appendix \ref{app:exp}.

\section{Relating the Heat and Caffarelli-Silvestre-Type Characterizations}
\label{sec:heat_CS}

In this section, we recall the relation between the heat kernel and the Caffarelli-Silvestre-type extension which had been obtained in \cite{ST10} and which we will use in the proof of Theorem \ref{thm:FGKU21} after recovering the (Neumann) Poisson kernel. In particular, we rely on a series expansion of the (Neumann) Poisson kernel which allows us to transfer information from $x_{n+1}=0$ to the slices $x_{n+1}>0$.

\begin{prop}[\cite{ST10}]
\label{prop:repr}
Let $(M,g)$ be a smooth, closed, connected manifold. Let
$f\in C^{\infty}(M)$ with $(f,1)_{L^2(M)} = 0$ and consider for $(x,x_{n+1}) \in M \times \R_+$
\begin{align*}
\tilde{u}(x,x_{n+1}) = \int\limits_{M} P_{x_{n+1}}(x,z) f(z) dz,
\end{align*}
where $P_{y}(x, z)$ denotes the Poisson kernel from \eqref{eq:Poisson}. Then $\tilde{u}\in \dot{H}^{1}(M \times \R_+ , x_{n+1}^{1-2s})$ and it defines the unique energy solution to \eqref{eq:Neu} with $(\tilde{u}(\cdot, x_{n+1}),1)_{L^2(M)}=0$ for all $x_{n+1} \in \R_+$.
\end{prop}

\begin{proof}
The proof of the validity of the equation satisfied by $\tilde{u}$ is discussed in \cite{ST10}; for completeness, the argument for the  $ \dot{H}^{1}(M \times \R_+ , x_{n+1}^{1-2s})$-regularity of $\tilde{u}$ can be found in Appendix \ref{app:exp}. 
\end{proof}

\begin{cor}
\label{cor:analytic}
Let $(M,g)$ be a smooth, closed, connected manifold of dimension $n\geq 2$.
Let $x,z \in M$ with $x\neq z $ and let $y\geq 0$. Then,
the integral in \eqref{eq:Poisson} converges absolutely and is analytic in $y\geq 0$. At $y=0$ the following series expansion holds
\begin{align}
\label{eq:Poisson2}
P_{y}(x,z):= \tilde{c}_s \sum\limits_{j=0}^{\infty} \left((-1)^j \int\limits_{0}^{\infty} K_{t}(x,z) \frac{4^{-j} t^{s-1-j}}{j!} dt \right) y^{2j}.
\end{align}
Its radius of convergence $r_0$ is strictly positive with a bound in terms of $ d_g(x,z)$ and the manifold $(M,g)$. A similar expansion with radius of convergence $r_{y_0}>0$ holds at any point $y_0 >0$. 
\end{cor}

\begin{proof}
The claims of the corollary strongly rely on the fact that $d_g(x,z)>0$, where we use $d_g(\cdot,\cdot)$ to denote the distance function with respect to the Riemannian metric $g$, and the estimates for the heat kernel on $(M,g)$ \cite{Gri95}:
\begin{align*}
|K_{t}(x,z)| \leq C_M t^{-\frac{n}{2}} e^{-c \frac{d_g(x,z)^2}{t}}, \ c>0.
\end{align*}
More precisely, for $n\geq 2$, by this heat kernel estimate, and since $x\neq z$, we have that for any $s\in (0,1)$, as a function in $t>0$,
\begin{align*}
K_{t}(x,z) e^{-\frac{y^2}{4t}} t^{s-1} = \sum\limits_{j=0}^{\infty} K_t(x,z) \frac{1}{4^j j!} (-1)^j t^{-j+s-1} y^{2j} \in L^1((\epsilon,\infty)),
\end{align*}
for any $\epsilon>0$.
For $0\leq y < 2 \sqrt{c} d_g(x,z)$, the function $2 e^{-c\frac{d_g(x,z)^2}{t}} t^{-\frac{n}{2}+s-1} e^{\frac{y^2}{4t}}$ is an (in $t$) integrable dominating function (for $x\neq z$ arbitrary but fixed), which hence implies that we may exchange summation and integration. As a consequence, for $\epsilon>0$ and $0\leq y < 2 \sqrt{c} d_g(x,z)$,
\begin{align}
\label{eq:series_o}
\begin{split}
 \sum\limits_{j= 0}^{\infty} \frac{4^{-j}(-1)^j}{j!}\left( \int\limits_{\epsilon}^{\infty} K_{t}(x,z)  t^{s-1-j} dt \right) y^{2j} = \int\limits_{\epsilon}^{\infty} K_{t}(x,z) e^{-\frac{y^2}{4t}} \frac{dt}{ t^{1-s}}.
 \end{split}
\end{align}
The integrals $c_j:= \int\limits_{\epsilon}^{\infty} K_{t}(x,z)  t^{s-1-j} dt$ arising in the coefficients of the expansion of the Neumann Poisson kernel satisfy
\begin{align*}
|c_j| \leq  \int\limits_{0}^{\infty} |K_{t}(x,z)|  t^{s-1-j} dt 
\leq  \int\limits_{0}^{\infty} e^{-\frac{c}{t}}  t^{s-1-j-n/2} dt  = c^{-\frac{n}{2}+s-j} \Gamma(\frac{n}{2}+4-s+j),
\end{align*}
where $c=c(x,z)\in (0,\infty)$, since $x\neq z$. Further, we recall that $\Gamma(j + 4+ \frac{n}{2}-s) \sim \Gamma(j) j^{\frac{n}{2}+4-s}$ as $j \rightarrow \infty$ \cite[equation (5.11.12)]{NIST}. As a consequence, for some $C>1$
\begin{align*}
\left|\frac{4^{-j}(-1)^j}{j!}\left( \int\limits_{\epsilon}^{\infty} K_{t}(x,z)  t^{s-1-j} dt \right) \right|
&\leq \left|\frac{4^{-j}(-1)^j}{j!}\left( \int\limits_{0}^{\infty} K_{t}(x,z)  t^{s-1-j} dt \right) \right| \\
&\leq C 4^{-j} c^{-\frac{n}{2}+s-j} j^{\frac{n}{2}+4-s}.
\end{align*}
Hence, for $|y| < 2 \sqrt{ c}$, the series in \eqref{eq:series_o} is dominated by a geometric series and the limit $\epsilon \rightarrow 0$ can be pulled into the series. Hence, \eqref{eq:Poisson2} is absolutely convergent at $y=0$ with positive radius of convergence.

For $y_0 >0$ we argue analogously and write 
\begin{align*}
K_{t}(x,z) e^{-\frac{y^2}{4t}} t^{s-1}
&= K_{t}(x,z) e^{-\frac{y^2_0 + 2y_0 (y-y_0)}{4t}} e^{-\frac{(y_0-y)^2}{4t}} t^{s-1}\\
& = \sum\limits_{j=0}^{\infty} K_t(x,z)  e^{-\frac{y^2_0 + 2y_0 (y-y_0)}{4t}} \frac{1}{4^j j!} (-1)^j t^{-j+s-1} (y-y_0)^{2j} \in L^1((\epsilon,\infty)),
\end{align*}
for any $\epsilon>0$. Again dominated convergence implies that for any $\epsilon>0$ and for $|y-y_0|$ sufficiently small,
\begin{align}
\label{eq:series_y0}
\begin{split}
&  \sum\limits_{j= 0}^{\infty} \frac{4^{-j}(-1)^j}{j!}\left( \int\limits_{\epsilon}^{\infty} K_{t}(x,z)  e^{-\frac{y^2_0 + 2y_0 (y-y_0)}{4t}} t^{s-1-j} dt \right) (y-y_0)^{2j}\\
&  = \int\limits_{\epsilon}^{\infty} K_{t}(x,z) e^{-\frac{y^2}{4t}} \frac{dt}{ t^{1-s}}, \ x,z\in M.
 \end{split}
\end{align}
Finally, we observe that for $j\geq 0$ and for $|y-y_0| < \max\{\frac{\sqrt{c}}{2},\frac{y_0}{4}\}$
\begin{align*}
 \int\limits_{\epsilon}^{\infty} |K_{t}(x,z) | e^{-\frac{y^2_0 + 2y_0 (y-y_0)}{4t}}   t^{s-1-j} dt 
&\leq  \int\limits_{0}^{\infty} e^{-\frac{c}{t}} e^{-\frac{y^2_0 + 2y_0 (y-y_0)}{4t}}  t^{s-1-j - \frac{n}{2}} dt  \\
&= \left( \frac{ 4c + y_0^2 + 2y_0 (y-y_0)}{4} \right)^{-\frac{n}{2}+s-j}  \Gamma(\frac{n}{2}+4-s+j)\\
& \leq \left( \frac{ c + y_0^2 }{8} \right)^{-\frac{n}{2}+s-j}  \Gamma(\frac{n}{2}+4-s+j).
\end{align*}
Hence, the series converges absolutely for $y_0 >0$ with positive radius of convergence.
\end{proof}

With the previous observations in place, we can deduce the equivalence between the Poisson and heat kernel data.

\begin{proof}[Proof of Proposition \ref{prop:equiv}]
The recovery of $P_{y}(x,z)$ from $K_{t}(x,z)$ is a direct consequence of the representation formula \eqref{eq:Poisson}. The converse is obtained by an inverse Fourier-Laplace transform. Indeed, after a change of coordinates, the identity \eqref{eq:Poisson} reads
\begin{align*}
P_y(x,z) = \int\limits_{0}^{\infty} \tilde{K}_{x,z}(r) e^{- \frac{y^2}{4} r} dr = \mathcal{L}(\tilde{K}_{x,z})(y^2/4),
\end{align*}
where $\tilde{K}_{x,z}(r):=-\tilde{c}_s K_{1/r}(x,z) r^{-1-s}$ and where $\mathcal{L}$ denotes the Laplace transform. By the heat kernel estimates from \cite{Gri95} and since $n\geq 1$, the function $\tilde{K}_{x,z}(r)$ is in $L^1$ and is  of exponential type, i.e., $|\tilde{K}_{x,z}(r)|\leq C e^{cr}$ for $r\geq 1$. As a consequence the inverse Laplace transform exists and is unique \cite[Chapter 2]{W15}. Thus, it is possible to recover $\tilde{K}_{x,z}(r)$ and hence $K_t(x,z)$ from the knowledge of $\mathcal{L}(\tilde{K}_{x,z})(y^2/4)$. This concludes the proof of the equivalence.
\end{proof}

\begin{rmk}
\label{rmk:alternative}
We sketch a second argument for the equivalence of the heat kernel from the localized (Neumann) Poisson kernel data: Using the series expansion
from \eqref{eq:Poisson2}, by iterative differentiation with respect to $y$ and evaluation at $y=0$, we recover all the coefficients 
\begin{align*}
a_j:=  \int\limits_{0}^{\infty} K_{t}(x,z)  t^{s-1-j} dt  = -\int\limits_{0}^{\infty} K_{1/r}(x,z)  r^{-s-1} r^{j} dr .
\end{align*}
As in \cite{FGKU21}, the knowledge of these can be viewed as the knowledge of all moments for $ K_{1/t}(x,z)  t^{-s-1}$. By the exponential decay in $t$ of $K_{1/t}(x,z)$, we obtain uniqueness as in \cite{FGKU21}: Indeed, in the notation of \cite{FGKU21}, the function $\varphi(r):= K_{1/r}(x,z)  r^{-s-1}$ satisfies $|\varphi(r)| \leq C e^{- c r}$ for some $c>0$ and $r\in [1,\infty)$. Moreover, for $r\in (0,1)$ we have that $|\varphi(r)| \leq C r^{\frac{n}{2}-1-s}$. As a consequence, since $n\geq 2$ and $s\in (0,1)$, we have $\varphi(r) \chi_{(0,\infty)}(r) \in L^1(\R)$, where $\chi_{(0,\infty)}(r)$ denotes the characteristic function of the interval $(0,\infty)$. The Fourier transform 
$\F(\chi_{(0,\infty)} \varphi)(\xi)$
is analytic for $|\Imm(\xi)|<c$ and satisfies
\begin{align*}
\F(\chi_{(0,\infty)} \varphi)(\xi)=\int\limits_{0}^{\infty} \varphi(r) e^{- i \xi r} dr = \sum\limits_{\ell = 0}^{\infty} \left(\frac{(-i)^{\ell}}{\ell!} \int\limits_{0}^{\infty} \varphi(r) r^{\ell} dr \right) \xi^{\ell} 
\end{align*}
for $|\xi|< c$. Here we used Lebesgue's dominated convergence theorem together with the estimates for $|\varphi|$.
Moreover, all the moments $a_j=\int\limits_{0}^{\infty} \varphi(r) r^{j} dr$, $j\in \N$, are known. Thus, if $a_j=0$ for all $j\in \N$, this implies that $\F(\chi_{(0,\infty)} \varphi)(\xi)=0$ for all $\xi \in \R$. As a consequence, $\varphi$ must vanish, which hence implies the uniqueness of the heat kernel, given the (Neumann) Poisson kernel. Furthermore, this argument is also constructive (up to analytic continuation) as the Fourier transform of the $L^1(\R)\subset \mathcal{S}'(\R)$ function $\chi_{(0,\infty)}(r) \varphi(r)$ can be recovered from the moments by Fourier inversion on $\mathcal{S}'(\R)$. Therefore the heat kernel can be constructively retrieved from the (Neumann) Poisson kernel.
\end{rmk}

\begin{rmk}
As in \cite{FGKU21} the knowledge of $P_{y}(x,z)$ with $x\in \omega_1, z \in \omega_2$ and $y>0$, varying over all $\omega_1, \omega_2 \subset O$ with the property that $\overline{\omega}_1 \cap \overline{\omega}_2 \neq \emptyset$ suffices to recover
$P_{y}(x,z)$ for all $x,z\in O$ and $y>0$. Indeed, this follows from a unique continuation argument. While in \cite{FGKU21} this was a unique continuation argument for solutions to the heat equation across spatial boundaries, in the setting of the (Neumann) Poisson kernel this follows from elliptic unique continuation arguments, see the proof of Proposition \ref{prop:Poisson_recov} below.
\end{rmk}

\section{Recovery of the Poisson Kernel}
\label{sec:Poisson_kernel}

In this section, we explain how to recover the Neumann and Dirichlet Poisson kernels from the source-to-solution data encoded in the operators $L_{s,O}$ and $\tilde{L}_{s,O}$, respectively. By the results from Section \ref{sec:heat_CS}, it then follows that the heat kernel can be recovered in $O$ (see also Remark \ref{rmk:upgrade} below). With the availability of the heat kernel information, the remainder of the proofs of Theorem \ref{thm:FGKU21} and Proposition \ref{prop:DtN} is then reduced to the same argument as in the proof from \cite{FGKU21} (first carrying out a Kannai transform to retrieve the wave kernel and then invoking the boundary control method). In order to avoid repetitions, we refer to \cite{FGKU21} for this and do not present this part of the argument.
An alternative approach which allows to directly recover the wave kernel from the Dirichlet Poisson kernel is discussed in Section \ref{sec:Kannai} below.

In order to highlight the dependence of the kernels on the underlying metrics, in this section we also write $P_{y}^{g}(x,z)$, $P_y^{D,g}(x,z)$ and $K_t^{g}(x,z)$ to denote the Neumann and Dirichlet Poisson and the heat kernels associated with a metric $g$.

We first discuss the setting of the Neumann-to-Dirichlet source-to-solution operator in Section \ref{sec:NtD} and then turn to the Dirichlet-to-Neumann version in the following Section \ref{sec:DtN}.

\subsection{The Neumann-to-Dirichlet problem}
\label{sec:NtD}
We begin by discussing the Neumann-to-Dirichlet setting. To this end, we show that it is possible to recover the kernel for disjoint tangential points $x,z \in M$ and $y\geq 0$ arbitrary.

\begin{prop}
\label{lem:Poisson_recov}
Let $(M,g)$ be a smooth, closed, connected manifold of dimension $n\geq 2$. Let $(O_1,g_1)\subset (M_1,g_1)$ and $(O_2,g_2)\subset (M_2,g_2)$ be open sets such that $(O_1,g_1)=(O_2,g_2)=:(O,g)$.
Let $s\in (0,1)$ and $L_{s,O}^{g_j}$ be as in \eqref{eq:sts1}, $j\in \{1,2\}$. Let $\omega_1, \omega_2 \subset O$ be open sets with $\overline{\omega}_1 \cap \overline{\omega}_2 = \emptyset$ and assume that $L_{s,O}^{g_1} f = L_{s,O}^{g_2} f$ for all $f\in C_c^{\infty}(\omega_2)$.
Then, for all $y\geq 0$ and for all $ x\in \omega_1, z\in \omega_2$ it holds that
\begin{align*}
P_{y}^{g_1}(x,z) = P_{y}^{g_2}(x,z).
\end{align*}
\end{prop}

\begin{proof}
We consider $f\in C_c^{\infty}(\omega_2)$ and denote by $u^f_j$ the solutions to \eqref{eq:source} with metrics $g_j$ and inhomogeneity $f$, satisfying $(u^f_j,1)_{L^2(M)}=0$ for $j\in \{1,2\}$, respectively. We use the notation $\tilde{u}^f_j$, $j\in\{1,2\}$, to refer to the associated Caffarelli-Silvestre-type extensions from \eqref{eq:Neu} with metrics $g_1, g_2$ and satisfying $(\tilde{u}_j^f(\cdot, x_{n+1}),1)_{L^2(M)}=0$ for all $x_{n+1}\geq 0$.
Then, by definition of the source-to-solution operator, it holds that 
\begin{align*}
\tilde{u}_1^f(x,0) = L_{s,O}^{g_1}f(x) = L_{s,O}^{g_2}f(x)= \tilde{u}_2^f(x,0) \mbox{ for all } f\in C_c^{\infty}(\omega_2) \mbox{ and } x\in \omega_1.
\end{align*}
Hence, by the representation formula from Proposition \ref{prop:repr}, since the above identity holds for all $f\in C_c^{\infty}(\omega_2)$, it follows that
\begin{align*}
P_{0}^{g_1}(x,z) = P_{0}^{g_2}(x,z),  \mbox{ for all } x\in \omega_1, z\in \omega_2.
\end{align*}
We seek to prove that 
\begin{align}
\label{eq:germ}
\p_y^{2m} P_{y}^{g_1}(x,z)|_{y=0} = \p_y^{2m} P_{y}^{g_2}(x,z)|_{y=0},  \mbox{ for all } x\in \omega_1, z\in \omega_2 \mbox{ and for all } m\in \N. 
\end{align}
Then, by the analyticity of the Poisson kernel in the $y$ variable and the fact that the expansion only depends on the even powers of $y$ (c.f. Corollary \ref{cor:analytic}), the claim follows.

We hence discuss the proof of \eqref{eq:germ}. To this end, we first note that the functions $\tilde{u}_1^f, \tilde{u}_2^f$ are $C^{\infty}$ regular in $\omega_1$. Indeed, this is a direct consequence of the representation \eqref{eq:Poisson} for the Neumann Poisson kernel, the regularity of the heat kernel and the fact that $\overline{\omega}_1 \cap \overline{\omega}_2 = \emptyset$. Approaching \eqref{eq:germ} we first consider the case $m=1$.
Since, by assumption, $u_1^f = u_2^f$ in $\omega_1$ and since $g_1 = g_2=g$ in $O$, by the representation formula for the Poisson kernel \eqref{eq:Poisson}, it follows that for some constant $c_2 =c_2(s)  \neq 0$ and for all $x\in \omega_1$
\begin{align*}
0 &= -\Delta_g (u_1^f- u_2^f)(x)
 = \lim\limits_{y \rightarrow 0} y^{2s-1} \p_y y^{1-2s} \p_y   (\tilde{u}_1^f- \tilde{u}_2^f)(x,y)\\
& = c_2 \int\limits_{O} (\p_y^{2}P_{y}^{g_1}(x,z) - \p_y^{2} P_{y}^{g_2}(x,z))\big|_{y=0}f(z)dz.
\end{align*}
Indeed, here we have used that by the regularity of $\tilde{u}_{j}^f$, $j\in \{1,2\}$, (which can be deduced from the representation formula in \eqref{eq:Poisson}) and the equation for these functions, we have for $x\in \omega_1$
\begin{align*}
-\Delta_g (u_1^f- u_2^f)(x)
& = -\lim\limits_{y\rightarrow 0}\Delta_g (\tilde{u}_1^f- \tilde{u}_2^f)(x,y)\\
&= \lim\limits_{y \rightarrow 0} y^{2s-1} \p_y y^{1-2s} \p_y   (\tilde{u}_1^f- \tilde{u}_2^f)(x,y)\\
&= \sum\limits_{j=1}^{\infty} \left[ \left( \int\limits_{O} \int\limits_{0}^{\infty} K_{t}(x,z) f(z) \frac{4^j (-1)^j}{ j!t^{1-s+j}}  dt \right) \lim\limits_{y\rightarrow 0} y^{2s-1}\p_y y^{1-2s} \p_y y^{2j} \right]\\
& = c_2\int\limits_{O} (\p_y^{2}P_{y}^{g_1}(x,z) - \p_y^{2} P_{y}^{g_2}(x,z))\big|_{y=0} f(z) dz.
\end{align*}
Since this holds for all $f\in C_c^{\infty}(\omega_2)$, we infer that 
\begin{align*}
(\p_y^{2}P_{y}^{g_1}(x,z) - \p_y^{2} P_{y}^{g_2}(x,z))\big|_{y=0}=0 \mbox{ for } x \in \omega_1, z \in \omega_2.
\end{align*}
By induction, an analogous argument also yields that the full statement of \eqref{eq:germ} holds for $x\in \omega_1$:
\begin{align*}
0 &= (-1)^m\Delta_g^m (u_1^f- u_2^f)(x)
 = (-1)^{m-1} \lim\limits_{y \rightarrow 0}  \Delta_g^{m-1} (y^{2s-1} \p_y y^{1-2s} \p_y)  (\tilde{u}_1^f- \tilde{u}_2^f)(x,y)\\ 
&  = \lim\limits_{y \rightarrow 0}    (y^{2s-1} \p_y y^{1-2s} \p_y)^{m}   (\tilde{u}_1^f- \tilde{u}_2^f)(x,y)\\
& = c_{2m,s} \int\limits_{O} (\p_y^{2m}P_{y}^{g_1}(x,z) - \p_y^{2m} P_{y}^{g_2}(x,z))\big|_{y=0} f(z)dz,
\end{align*}
with $c_{2m,s}\neq 0$.
By the arbitrary choice of  $f\in C_c^{\infty}(\omega_2)$, this indeed proves \eqref{eq:germ}.
\end{proof}

While strictly speaking this is not necessary on the level of the Neumann Poisson kernel but can still later be deduced on the level of the heat kernel (see Remark \ref{rmk:upgrade} below), we nevertheless upgrade the previous result to all points $x,z \in O$ which are not necessarily disjoint.

\begin{cor}
\label{prop:Poisson_recov}
Let $(M,g)$ be a smooth, closed, connected manifold of dimension $n\geq 2$. Let $(O_1,g_1)\subset (M_1,g_1)$ and $(O_2,g_2)\subset (M_2,g_2)$ be open sets such that $(O_1,g_1)=(O_2,g_2)=:(O,g)$.
Let $s\in (0,1)$ and $L_{s,O}^{g_j}$ be as in \eqref{eq:sts1}, $j\in \{1,2\}$. Assume that $L_{s,O}^{g_1} f = L_{s,O}^{g_2} f$ for all $f\in C_c^{\infty}(O)$. 
Then, for all $y> 0$ and for all $x,z \in O$ it holds that
\begin{align*}
P_{y}^{g_1}(x,z) = P_{y}^{g_2}(x,z).
\end{align*}
\end{cor}

\begin{proof}
For $y>0$ the functions $P_{y}^{g_1}(x,z), P_{y}^{g_2}(x,z)$ are solutions to the bulk equations associated with the Caffarelli-Silvestre-type extension problem in the variables $(x,y)$ and $(z,y)$ respectively. Hence, for $y>0$, by (uniformly) elliptic unique continuation \cite{KT03}, the equality from Proposition \ref{lem:Poisson_recov} and by $\omega_1, \omega_2 \subset O$ being arbitrary, we infer that
\begin{align*}
P_{y}^{g_1}(x,z) = P_{y}^{g_2}(x,z),  \mbox{ for all } x,z \in O.
\end{align*}
\end{proof}

\begin{rmk}
\label{rmk:upgrade}
With the results of Propositions \ref{prop:equiv} and \ref{lem:Poisson_recov} in hand, it is possible to recover the kernel $K_t(x,z)$ for all $x,z\in O$ with $x\neq z$ and $t> 0$. As in \cite{FGKU21}, a spatial unique continuation argument then yields that $K_t(x,z)$ can be recovered for all $x,z \in O$ and $t>0$. Indeed, for $t>0$, the kernel $K_t(x,z)$ is regular in $x,z,t$ and hence solves the heat equation in the variables $(t,x)$ and $(t,z)$, respectively. By spatial unique continuation for the heat equation, the identity  $K_t^{g_1}(x,z) = K_t^{g_2}(x,z)$ for $x,z\in O$ with $x\neq z$ and $t> 0$ then implies that $K_t^{g_1}(x,z) = K_t^{g_2}(x,z)$ for all $x,z\in O$ and $t> 0$.
\end{rmk}

The remainder of the proof of Theorem \ref{thm:FGKU21} then follows exactly as in \cite{FGKU21}. To avoid redundancies, we do not repeat this here.

\subsection{The Dirichlet-to-Neumann setting}
\label{sec:DtN}

In this section, we discuss the short argument for Proposition \ref{prop:DtN}. As in the Neumann-to-Dirichlet setting, we first determine the associated generalized Dirichlet Poisson kernel \eqref{eq:PoissonD} on $O$ and then use the moment problem together with the Fourier transform to infer the heat kernel information on $O$. 

\begin{proof}[Proof of Proposition \ref{prop:DtN}]
It suffices to recover the heat kernel information in $O$, i.e., to prove that for all $t>0$ and all $x,z \in O$
\begin{align*}
K_t^{g_1}(x,z) = K_t^{g_2}(x,z).
\end{align*}
The remainder of the argument then follows from \cite{FGKU21}. An alternative approach, bypassing the heat kernel, which directly connects the Dirichlet Poisson kernel and the wave kernel, will be presented in Section \ref{sec:Kannai} below.

As in the proof of the reconstruction of the heat kernel from the Neumann-to-Dirichlet based source-to-solution operator \eqref{eq:sts1}, also in the Dirichlet-to-Neumann case we infer the heat kernel information from the source-to-solution operators $\tilde{L}_{s,O}^{g_j}$, $j\in \{1,2\}$, from \eqref{eq:sts2} by recovering the identity
\begin{align*}
P_y^{D,g_1}(x,z) = P_y^{D,g_2}(x,z) \mbox{ for } y >0 \mbox{ and } x,z\in O, \ x\neq z,
\end{align*}
where $P_y^{D,g_j}(x,z)$ denotes the (Dirichlet) Poisson kernel associated with the metric $g_j$, $j\in \{1,2\}$.

To this end, similarly as in the Neumann setting, we consider $\omega_1, \omega_2 \subset O$ with $\overline{\omega_1} \cap \overline{\omega_2} = \emptyset$ and such that $x\in \omega_1, z\in \omega_2$. By the results in \cite{ST10} (see also \cite{CS16} and Appendix \ref{app:exp}) we then have that, for $f\in C_c^{\infty}(\omega_1)$, the solution $\tilde{v}_j \in \dot{H}^1(M_j \times \R_+, x_{n+1}^{1-2s})$ of \eqref{eq:Dir} with metric $g_j$ can be expressed as
\begin{align*}
\tilde{v}_j(x,y)= c_s  y^{2s } \int\limits_{M_j}\int\limits_{0}^{\infty} K_t^{g_j}(x,z) e^{-\frac{y^2}{4t}} \frac{dt}{t^{1+s}} f(z) dz, \ j\in\{1,2\}.
\end{align*}
Here $K_t^{g_j}(x,z)$ denotes the heat kernel on $(M_j, g_j)$. By the assumption that $\overline{\omega_1} \cap \overline{\omega_2} = \emptyset$ and the heat kernel estimates, we obtain that $\int\limits_{0}^{\infty} K_t^{g_j}(x,z) e^{-\frac{y^2}{4t}} \frac{dt}{t^{1+s}} $ is absolutely convergent for $x\in \omega_1$, $z\in \omega_2$ and that, similarly as in the Neumann setting, it can be expanded into a series in $y \geq 0$. In particular, the series expansion at $y=0$ reads
\begin{align}
\label{eq:D_expansion}
y^{2s}\int\limits_{0}^{\infty} K_t^{g_j}(x,z) e^{-\frac{y^2}{4t}} \frac{dt}{t^{1+s}}
= \sum\limits_{\ell=0}^{\infty} c_{\ell}^{g_j}(x,z) y^{2s + 2\ell},
\end{align}
with 
\begin{align}
\label{eq:coefDtN}
c_{\ell}^{g_j}(x,z)= \frac{1}{4^{\ell} \ell!}(-1)^{\ell} \int\limits_{0}^{\infty} K_t^{g_j}(x,z) \frac{1}{t^{1+s+\ell}} dt, \ j\in\{1,2\}.
\end{align} 
As in the Neumann case this series has a positive radius of convergence, depending only on $d_g(x,z)$ and the manifold $(M,g)$ (through the heat kernel bounds). For points $y_0 >0$ by a similar argument as in the case of the Neumann kernel, also the Poisson kernel is analytic.

In particular, choosing $y>0$ sufficiently small (so that the series expansion at $y=0$ is absolutely convergent), this implies that for $x\neq z$, the limit $\lim\limits_{y \rightarrow 0} y^{1-2s}\p_y \tilde{v}_j(x,y)= \tilde{L}_{s,O}^{g_j}(f)$ exists and determines $c_0^{g_j}$, $j\in\{1,2\}$. Here, as above, the functions $\tilde{v}_j$, $j\in\{1,2\}$, denote the Caffarelli-Silvestre-type extensions (c.f. \eqref{eq:Dir}) of the data $f$ with respect to the metrics $g_j$, $j\in\{1,2\}$.

Now, the knowledge of the Dirichlet-to-Neumann source-to-solution operator $\tilde{L}_{s,O}^{g_j}$ applied to data $f\in C_c^{\infty}(\omega_2)$ and $x\in \omega_1$ leads to the following identity 
\begin{align*}
&\lim\limits_{y \rightarrow 0} y^{1-2s}\p_y\tilde{v}_1(x,y) = \int\limits_{M_1}c_0^{g_1}(x,z)f(z) dz 
= \int\limits_{M_2}c_0^{g_2}(x,z)f(z) dz = \lim\limits_{y \rightarrow 0} y^{1-2s}\p_y\tilde{v}_2(x,y).
\end{align*} 
Since $f\in C_c^{\infty}(\omega_2)$ was arbitrary, this yields that $c_0^{g_1}(x,z) = c_0^{g_2}(x,z)$ for all $x \in \omega_1, z\in \omega_2$.

Furthermore, we also have that for all $m\in \N$
\begin{align*}
(-\Delta_{g})^m \lim\limits_{y \rightarrow 0} y^{1-2s}\p_y  (\tilde{v}_1-\tilde{v}_2)(x,y) = 0 \mbox{ for } x\in \omega_1,
\end{align*}
where we have used the assumption that $g_1 = g_2 =g$ in $O$.
By the equations for the functions $\tilde{v}_j$, $j\in\{1,2\}$, in $O\times \R_+$, this however corresponds to 
\begin{align*}
0&=(-\Delta_{g})^m \lim\limits_{y \rightarrow 0} y^{1-2s}\p_y (\tilde{v}_1-\tilde{v}_2)(x,y)\\
&= \lim\limits_{y\rightarrow 0}(-\D_g)^m y^{1-2s}\p_y  (\tilde{v}_1-\tilde{v}_2) (x,y)\\
&= \lim\limits_{y\rightarrow 0}(y^{2s-1}\p_y y^{1-2s}\p_y)^{m} y^{1-2s}\p_y (\tilde{v}_1-\tilde{v}_2)(x,y)\\
&= c_{s,m} \int\limits_{O} (c_m^{g_1}(x,z)-c_m^{g_2}(x,z)) f(z) \mbox{ for } x\in \omega_1.
\end{align*}
As above, this hence results in 
\begin{align*}
c_m^{g_1}(x,z)-c_m^{g_2}(x,z) = 0 \mbox{ for all } x \in \omega_1,  z \in \omega_2, \ m\in \N.
\end{align*}
Therefore, by \eqref{eq:D_expansion} and \eqref{eq:coefDtN}, by a moment argument together with the Fourier transform (see Remark \ref{rmk:alternative}), as in the Neumann setting, we deduce that for $t>0$ and for all $x\in \omega_1, \ z\in \omega_2$,
\begin{align*}
K_t^{g_1}(x,z) = K_t^{g_2}(x,z).
\end{align*}
Since $x,z\in O$ with $x\neq z$ were arbitrary, this implies that for $t>0$ 
\begin{align*}
K_t^{g_1}(x,z) = K_t^{g_2}(x,z) \mbox{ for all } x,z \in O, \ x\neq z.
\end{align*}
\end{proof}

Concluding this section, we remark that at least on a formal level, the result of Proposition \ref{prop:DtN} can be recovered from Theorem \ref{thm:FGKU21}. 

\begin{rmk}
\label{rmk:CR}
In order to observe the reduction, let $\tilde{v}_j\in \dot{H}^{1}(M_j\times \R_+, x_{n+1}^{1-2s})$ be the solution of \eqref{eq:Dir} with boundary data $f\in C^{\infty}(M_j)$ and $(f,1)_{L^2(M_j)}=0$ and metric $g_j$, $j\in \{1,2\}$. For $j\in\{1,2\}$ we define the function
\begin{align}
\label{eq:integrated}
\tilde{u}_j(x,x_{n+1}):= -\int\limits_{x_{n+1}}^{\infty} t^{1-2s} \tilde{v}_j(x,t) dt,
\end{align}
and note that it is well-defined (i.e. the integral exists and is finite) and satisfies for some $\bar{c}_s \neq 0$ 
\begin{align}
\label{eq:CS_new}
\begin{split}
(x_{n+1}^{2s-1} \D_{g_j} + \p_{n+1} x_{n+1}^{2s-1} \p_{n+1}) \tilde{u}_j & = 0 \mbox{ in } M_j \times \R_+,\\
\bar{c}_s \lim\limits_{x_{n+1}\rightarrow 0} x_{n+1}^{2s-1}\p_{n+1} \tilde{u}_j(x,x_{n+1}) &= f \mbox{ on } M_j \times \{0\}.
\end{split} 
\end{align}
In particular, $\tilde{u}_j(x,0) = \tilde{c}_s(-\D_{g_j})^{s-1}f(x)$ for $x\in M_j$, $j\in\{1,2\}$ and some constant $\tilde{c}_s \neq 0$.

To observe this, we note that on a closed, smooth, connected manifold $(M,g)$ a general solution $\tilde{v}\in \dot{H}^{1}(M\times \R_+, x_{n+1}^{1-2s})$ of \eqref{eq:Dir} can be written in terms of an eigenfunction expansion as
\begin{align*}
\tilde{v}(x,x_{n+1}) = c_0(1,f)_{L^2(M)}1 + \sum\limits_{k=1}^{\infty} c_k(x_{n+1}) \phi_k(x),
\end{align*}
with $c_k(t):= c_{s} (\sqrt{\lambda_k} t)^{s} (f,\phi_k)_{L^2(M)} K_s(\sqrt{\lambda_k} t)$, where $K_s$ denotes a modified Bessel function of the second kind (see the computations in \cite[Section 3]{ST10}, in \cite[Appendix A]{GFR19}, the definitions in \cite{NIST} or the computations in Appendix \ref{app:exp}). The functions $\{\phi_k\}_{k\in \N}$ denote an $L^2$-normalized eigenfunction basis of the Laplace-Beltrami operator on the respective manifold with eigenvalues $\{\lambda_k\}_{k\in \N}$. Now, recalling the exponentially decaying asympotics of the modified Bessel functions of the second kind at infinity, if $(1,f)_{L^2(M)} = 0$, all the coefficients $c_k(t)$ have sufficient decay to ensure that the quantity \eqref{eq:integrated} is well-defined. We point to the proof of Lemma \ref{lem:relations} for a similar argument.
The validity of the equation \eqref{eq:CS_new} then follows from a direct computation. 

Last but not least, it thus remains to connect the data for the Dirichlet and Neumann source-to-solution operators.
To this end, we note that  for $x\in O$ and $j\in\{1,2\}$
\begin{align*}
L_{s-1,O}^j f(x)=\tilde{u}_j(x,0) = - \int\limits_{0}^{\infty} t^{1-2s} \tilde{v}_j(x,t)dt = \tilde{c}_s (-\D_{g_j})^{s-1} f(x) = \tilde{c}_s (-\D_{g_j})^{-1} \tilde{L}_{s,O}^j f(x).
\end{align*}
Hence, from the fact that 
\begin{align*}
\tilde{L}^{1}_{s,O} = \tilde{L}^{2}_{s,O},
\end{align*}
the definition $\tilde{L}^{j}_{s,O} f = (-\D_{g_j})^s f|_{O}$, $j\in\{1,2\}$, the locality of $(-\D_g)$ as well as the fact that $g= g_1 =g_2$ in $O$,
we obtain that
\begin{align*}
(-\D_{g_1})L^{1}_{s-1,O} = (-\D_{g_2}) L^{2}_{s-1,O}.
\end{align*}
Since the operators $(-\D_{g_j}), L^{j}_{s-1,O}$, $j\in\{1,2\}$, commute (which can, for instance, be observed by expanding the identity in terms of Laplace-Beltrami eigenfunctions) this results in the identity
\begin{align*}
L^{1}_{s-1,O} (-\D_{g_1}) = L^{2}_{s-1,O} (-\D_{g_2}) .
\end{align*}
By the surjectivity of the Laplacian as a map from $C^{\infty}(M)$ with vanishing mean value to $C^{\infty}(M)$ with vanishing mean value we deduce that 
\begin{align*}
L^{1}_{s-1,O}  = L^{2}_{s-1,O} .
\end{align*}
This concludes the argument.
\end{rmk}

\section{A Direct Kannai-Type Transform for the Dirichlet Poisson Operator}
\label{sec:Kannai}

In this section, we discuss a Kannai-type transform for the Dirichlet Poisson operator for the fractional Laplacian. This allows us to directly transfer information from the Dirichlet Poisson kernel to the wave kernel on $(M,g)$. In particular, while being accessible, it is not necessary to pass through the heat kernel in the proof of Proposition \ref{prop:DtN}.

We begin with the proof of Proposition \ref{prop:Kannai}.

\begin{proof}[Proof of Proposition \ref{prop:Kannai}]
We deduce the desired identities on the level of an eigenfunction expansion. To this end, we consider $f\in C^{\infty}(M)$ and its associated variable coefficient Caffarelli-Silvestre-type extension $\tilde{v}$ which, in particular, satisfies \eqref{eq:Dir}.

\emph{Step 1: Reduction to an eigenfunction expansion.} A short computation as, for instance, in \cite[Chapter 3]{ST10},  \cite[Appendix A]{GFR19} or in Appendix \ref{app:exp}, allows to infer that for $f\in C^{\infty}(M)$ and $x\in M$
\begin{align*}
%\tilde{u}(x,y)&:=\int\limits_{M} P_{y}(x,z)f(z) dz = \sum\limits_{k=0}^{\infty} c_{k,s}(y) \phi_k(x), \mbox{ for } (f,1)_{L^2(M)}=0,\\
\tilde{v}(x,y)&:=\int\limits_{M} P_{y}^D(x,z)f(z) dz = \sum\limits_{k=0}^{\infty} d_{k,s}(y) \phi_k(x),
\end{align*}
with $d_{k,s}(y) = c_{s,D} (\sqrt{\lambda_k} y)^s K_s(\sqrt{\lambda_k} y) (f, \phi_k)_{L^2(M)}$. 
%and $c_{k,s}(y) = K_{s}(\sqrt{\lambda} y ) \left(\frac{y}{\sqrt{\lambda}} \right)^{s}$. 
Here $\{\phi_k\}_{k\in \N}$, $\{\lambda_k\}_{k\in \N}$ denote an $L^2(M)$-normalized eigenbasis with associated eigenvalues of the Laplace-Beltrami operator on $(M,g)$.
Given this, we consider an expansion of $f\in C^{\infty}(M)$ into this basis:
\begin{align*}
f(x) = \sum\limits_{\ell =0}^{\infty} (f,\phi_{\ell})_{L^2(M)} \phi_{\ell}(x), \ x\in M.
\end{align*}
Using the expression for the wave and Poisson kernels in the eigenbasis, for $x\in M$ the desired identity then reads
\begin{align*}
&\sum\limits_{k=0}^{\infty} c_{s,D} (\sqrt{\lambda_k} y)^s K_s(\sqrt{\lambda_k} y) (f,\phi_k)_{L^2(M)} \phi_k(x)\\
&= \sum\limits_{k=0}^{\infty} \bar{c}_{s,D} y^{2s} \int\limits_{0}^{\infty} (\tau + y^2)^{-\frac{3}{2}-s} \frac{\sin(\sqrt{\tau} \sqrt{\lambda_k})}{\sqrt{\lambda_k}} d \tau (f,\phi_k)_{L^2(M)} \phi_k(x)
\end{align*}
for some constants $c_{s,D}, \bar{c}_{s,D} \neq 0$.
If we can prove that for $\lambda \geq 0$
\begin{align}
\label{eq:eigen_id}
 \bar{c}_{s,D} y^{2s} \int\limits_{0}^{\infty} (\tau + y^2)^{-(\frac{3}{2}+s)}  \frac{\sin(\sqrt{\lambda} \tau)}{\sqrt{\lambda}} d\tau
 = c_{s,D} (\sqrt{\lambda} y)^s K_s(\sqrt{\lambda} y),
\end{align}
then, by orthogonality, it holds that for some constant $\bar{c}_{s,D}$ and for all $x\in O$ and all $f\in C^{\infty}(M)$
\begin{align*}
\int\limits_{M} P_y^D(x,z) f(z)dz = \int\limits_{M} \bar{c}_{s,D} y^{2s} \int\limits_{0}^{\infty} (\tau + y^2)^{-\frac{3}{2}-s} \frac{\sin(\sqrt{\tau} \sqrt{(-\D_g)})}{\sqrt{(-\D_g)}}(x,z) d \tau f(z) dz.
\end{align*}
As a consequence, this implies the identity of Proposition \ref{prop:Kannai}.

\emph{Step 2: Proof of the representation formula \eqref{eq:eigen_id}.}
By the previous step, it suffices to only prove the statement of \eqref{eq:eigen_id} on the individual eigenfunctions and eigenvalues. By a change of coordinates (passing from $\sqrt{\tau}$ to $\tau$), it suffices to prove that for $\lambda \geq 0$
\begin{align*}
 \bar{c}_{s,D} y^{2s} \int\limits_{0}^{\infty} (\tau^2 + y^2)^{-(\frac{3}{2}+s)} \tau \frac{\sin(\sqrt{\lambda} \tau)}{\sqrt{\lambda}} d\tau
 = c_{s,D} (\sqrt{\lambda} y)^s K_s(\sqrt{\lambda} y).
\end{align*}
This identity, in turn, follows by firstly observing that, by integration by parts, for $y>0$ and $\lambda \geq 0$,
\begin{align*}
 \bar{c}_{s,D} \int\limits_{0}^{\infty} (\tau^2 + y^2)^{-(\frac{3}{2}+s)} \tau \frac{\sin(\sqrt{\lambda}\tau)}{\sqrt{\lambda}} d\tau
 = \tilde{c}_{s,D} \int\limits_{0}^{\infty} (\tau^2 + y^2)^{-(\frac{1}{2}+s)}  \cos(\sqrt{\lambda} \tau) d\tau.
\end{align*}
Secondly, for $\lambda>0$ the expression on the right hand side is a form of Basset's Integral and is known to be related to the modified Bessel function of the second type \cite[formula 10.32.11]{NIST}
\begin{align*}
\int\limits_{0}^{\infty} (\tau^2 + y^2)^{-(\frac{1}{2}+s)} \cos(\sqrt{\lambda} \tau) d\tau
= c(s) K_s(\sqrt{\lambda} y ) y^{-s} \sqrt{\lambda}^s.
\end{align*}
As a consequence, for $\lambda > 0$
\begin{align*}
 \bar{c}_{s,D} y^{2s} \int\limits_{0}^{\infty} (\tau^2 + y^2)^{-(\frac{3}{2}+s)} \tau \frac{\sin(\sqrt{\lambda} \tau)}{\sqrt{\lambda}} d\tau
 = c_{s,D} K_s(\sqrt{\lambda} y) (\sqrt{\lambda} y)^s,
\end{align*}
for some constants $c_{s,D}, \bar{c}_{s,D}\neq 0$. 
Finally, for $\lambda = 0$, we note that for some constant $c_s \neq 0$
\begin{align*}
\int\limits_{0}^{\infty} (\tau^2 + y^2)^{-( \frac{1}{2} +s) } d\tau 
&= y^{-(1+2s)} \int\limits_{0}^{\infty} \left(1 + \frac{\tau^2}{y^2} \right)^{-(\frac{1}{2}+s)} d\tau
= y^{-2s} \int\limits_{0}^{\infty} (1 + r^2)^{-(\frac{1}{2}+s)} dr\\
&= c_s y^{-2s}, \ y\geq 0.
\end{align*}
Invoking the asymptotics of the modified Bessel functions \cite[Chapter 10.30]{NIST} this also implies the claim for $\lambda=0$.
Combining the above observations for all $\lambda \geq 0$, by orthogonality, then implies the desired identity for $P_{y}^D(x,z)$.
\end{proof}

In order to use the relation between the Poisson kernel and the wave equation, we next show that it is possible to invert the identity from Proposition \ref{prop:Kannai} and to recover the wave kernel from the heat kernel.

\begin{prop}[Inversion of the Kannai transforms]
\label{prop:Kannai_inversion}
Let $(M,g_j)$, $j\in\{1,2\}$, be smooth, closed, connected manifolds. Let $s\in (0,1)$. Let $y>0$, let $x,z\in M$ and let $P_{y}^{D,g_j}(x,z)$ be the Dirichlet Poisson kernels and let $K_{\tau}^{w,g_j}(x,z):= \frac{\sin(\tau\sqrt{-\D_{g_j}})}{\sqrt{-\D_{g_j}}}(x,z)$, $j\in \{1,2\}$, denote the wave kernels. Then, it is possible to invert the relation from Proposition \ref{prop:Kannai}. More precisely, if for all $f\in C^{\infty}_c(O)$ and all $x\in O$ it holds that
\begin{align*}
\int\limits_{O} P_{y}^{D,g_1}(x,z) f(z) dz = \int\limits_{O} P_{y}^{D,g_2}(x,z) f(z) dz \mbox{ for all } y\geq 0,
\end{align*}
then it follows that for all $f\in C_c^{\infty}(O)$ and $x\in O$
\begin{align*}
\int\limits_{O} K_{\sqrt{\tau}}^{w,g_1}(x,z) f(z) dz = \int\limits_{O} K_{\sqrt{\tau}}^{w,g_2}(x,z) f(z) dz \mbox{ for all } \tau \geq 0.
\end{align*}
Moreover, this inversion is constructive.
\end{prop}

\begin{proof}
We consider $f\in C^{\infty}_c(O)$. Since, independently of the metrics $g_1, g_2$, (up to normalization) the zeroth eigenfunction is constant equal to one, $\phi_0 \equiv 1$, and since the claim holds on this mode, we may without loss of generality assume that $(f,1)_{L^2(O)} = 0$. For later use, we record that on the $L^2(M)$ orthogonal complement of the function $f=1$ (which, by the support assumption of $f$, is equal to the $L^2(O)$ complement of $f=1$ which only depends on the known metric $g$), an eigenfunction expansion shows that the kernels $K_{\tau}^{g_j}(x,z)$, $j\in\{1,2\}$, are bounded as operators from $L^2(M)$ to $L^2(M)$ independently of $\tau$.

Now, by assumption and by Proposition \ref{prop:Kannai} we have that for all $x\in O$ and $y>0$
\begin{align*}
\int\limits_{O} \int\limits_{0}^{\infty} K_{\sqrt{\tau}}^{w,g_1}(x,z) (\tau + y^2)^{-(\frac{3}{2}+s)} d\tau f(z) dz
= \int\limits_{O} \int\limits_{0}^{\infty} K_{\sqrt{\tau}}^{w,g_2}(x,z) (\tau + y^2)^{-(\frac{3}{2}+s)} d\tau f(z) dz .
\end{align*}
Using dominated convergence (for which we recall the $L^2(M) \rightarrow L^2(M)$ bound for the kernels $K_{\tau}^{g_j}(x,z)$, $j\in\{1,2\}$) to exchange differentiation in $y$ and the $\tau,z$ integration, we differentiate both sides $\ell$-times with respect to $y$ and evaluate all expressions at $y=1$. This yields that for all $\ell\in \N$, $x\in O$
\begin{align*}
\int\limits_{O} \int\limits_{0}^{\infty} K_{\sqrt{\tau}}^{w,g_1}(x,z) (\tau + 1)^{-(\frac{3}{2}+s+\ell)} d\tau f(z) dz
= \int\limits_{O} \int\limits_{0}^{\infty} K_{\sqrt{\tau}}^{w,g_2}(x,z) (\tau + 1)^{-(\frac{3}{2}+s +\ell)} d\tau f(z) dz .
\end{align*}
Shifting the variable by setting $r = \tau + 1$, setting $\tilde{h}_{x,z}^{g_j}(r):= r^{-(\frac{3}{2}+s+2)}K_{\sqrt{r-1}}^{w,g_j}(x,z)$, $j\in \{1,2\}$, and defining $t= \frac{1}{r}$ results in 
\begin{align*}
\int\limits_{O} \int\limits_{0}^{1} \tilde{h}_{x,z}^{g_1}(1/t) t^{\ell} dt f(z) dz
= \int\limits_{O} \int\limits_{0}^{1} \tilde{h}_{x,z}^{g_2}(1/t) t^{\ell} dt f(z) dz .
\end{align*}
Setting $h_{x,z}^{g_j}(t) = \tilde{h}_{x,z}^{g_j}(1/t)$, $j\in\{1,2\}$, this is equivalent to the equality of all moments of the functions $h_{x,z}^{g_j}(t) $
\begin{align}
\label{eq:moments}
\int\limits_{O} \int\limits_{0}^{1} h_{x,z}^{g_1}(t) t^{\ell} dt f(z) dz
= \int\limits_{O} \int\limits_{0}^{1} h_{x,z}^{g_2}(t) t^{\ell} dt f(z) dz .
\end{align}
From this
we deduce that for $\xi \in \R$
\begin{align*}
\mathcal{F}(\chi_{(0,1)}(\cdot) \int\limits_O h_{x,z}^{g_1}(\cdot) f(z) dz)(\xi)
&= \int\limits_{O} \int\limits_{0}^{1} h_{x,z}^{g_1}(t) e^{-i\xi t} dt f(z) dz  \\
&= \int\limits_{O} \int\limits_{0}^{1} h_{x,z}^{g_2}(t) e^{-i\xi t}  dt f(z) dz \\
&= \mathcal{F}(\chi_{(0,1)}(\cdot) \int\limits_O h_{x,z}^{g_2}(\cdot) f(z) dz)(\xi).
\end{align*}
Indeed, for $j\in \{1,2\}$, we have that
\begin{align}
\label{eq:FT}
\mathcal{F}(\chi_{(0,1)}(\cdot) \int\limits_O h_{x,z}^{g_j}(\cdot) f(z) dz)(\xi)
= \sum\limits_{\ell =0}^{\infty} \left(\frac{(-i)^{\ell}}{\ell!} \int\limits_{O} \int\limits_{0}^{1} h_{x,z}^{g_j}(t) t^{\ell} dt f(z) dz \right) \xi^{\ell}, \ j\in\{1,2\},
\end{align}
where we have used \eqref{eq:moments} together with Lebesgue's dominated convergence theorem to interchange summation and integration for which we record that for $x,z \in M$ the map
\begin{align*}
 t \mapsto \int\limits_O h_{x,z}^{g_j}(t) f(z) dz, \ j\in\{1,2\},
\end{align*}
is integrable for $t\in (0,1)$.
Since 
\begin{align*}
 t \mapsto \chi_{(0,1)}(t)\int\limits_O h_{x,z}^{g_j}(t) f(z) dz, \ j\in\{1,2\},
\end{align*}
is an $L^2(\R)$ function, Fourier inversion implies the claim.
\end{proof}

\begin{rmk}
We emphasize that the above proof of Proposition \ref{prop:Kannai_inversion} is constructive: Given a Poisson kernel, it is possible to compute the moments as in \eqref{eq:moments} and the Fourier transform analogously as in \eqref{eq:FT}. This can then be inverted.
\end{rmk}

Hence, as in \cite{FGKU21}, it is now possible to recover the source-to-solution operators $L^{\text{wave}}_{g_j,O}$, $j\in \{1,2\}$, for the wave equation from the source-to-solution operators $\tilde{L}^j_{s,O}$ for the Dirichlet problem. Indeed, consider $F\in C_c^{\infty}((0,\infty) \times O)$ and let $u \in C^{\infty}([0,\infty) \times M)$ be a solution to 
\begin{align*}
(\p_t^2 - \Delta_g) u& = F \mbox{ on } (0,\infty)\times M,\\
u, \p_t u & = 0 \mbox{ on } \{0\}\times M.
\end{align*}
The source-to-solution map for the wave operator is then given by the map
\begin{align}
\label{eq:wavests}
L^{\text{wave}}_{g,O}: C_c^{\infty}((0,\infty) \times O) \rightarrow C^{\infty}([0,\infty) \times O), \ 
L^{\text{wave}}_{g,O}(F) = u|_{O},
\end{align}
with the solution $u(t,x)$ being expressed by means of the Duhamel formula (see Remark \ref{rmk:Duhamel} from above) 
\begin{align*}
u(t,x)= \int\limits_{0}^t \int\limits_{M} \frac{\sin((t-s)\sqrt{(-\D_g)})}{\sqrt{(-\D_g)}}(x,z) F(s,z) dz ds \mbox{ for } x\in O, t\geq 0.
\end{align*}
Since the previous results in this section prove that it is possible to recover the wave kernel restricted to data and measurements in $O$ from the source-to-solution map of the Dirichlet problem for the fractional Laplacian in $O$, the maps $L^{\text{wave}}_{g_1,O}$ and $L^{\text{wave}}_{g_2,O}$ are known and are equal on $F\in C_c^{\infty}((0,\infty)\times O)$, $j\in \{1,2\}$. We are thus in the same setting as in the last step in \cite{FGKU21} in which the boundary control method can be used to recover the metric up to diffeomorphisms fixing the set $O$. This hence concludes the proof of Proposition \ref{prop:DtN} without using the heat kernel realization as an intermediate step.

\begin{rmk}
We highlight that, as outlined in this section, in the setting of Proposition \ref{prop:DtN} it is thus \emph{not} necessary to argue through the auxiliary step of recovering the heat kernel. By virtue of Proposition \ref{prop:Kannai_inversion} it is possible to \emph{directly} recover the wave kernel in $O$ and hence directly conclude the proof of the uniqueness result of Proposition \ref{prop:DtN} by using the boundary control method as in \cite{FGKU21}.
\end{rmk}

\section{Relating the Local and Nonlocal Source-to-Solution Anisotropic Calder\'on Problems}
\label{sec:relation}

In this final section, we relate the local and nonlocal anisotropic Calder\'on problems with source-to-solution data on smooth, closed, connected manifolds $(M,g)$ of dimension $n\geq 2$. Let $O \subset M$ and let $s\in (0,1)$. In this section, we use both the spectral and Caffarelli-Silvestre-type characterizations of the fractional Laplacian.

In order to formulate the reduction of the nonlocal problem to the local setting, we introduce the following Cauchy data sets. We use the notation $\mathcal{C}_s$ for the nonlocal, $\mathcal{C}_1$ for the local and $\mathcal{\tilde{C}}_1$ for an auxiliary local problem which, as shown below, can be recovered from the nonlocal data:
\begin{align*}
\mathcal{C}_s &:= \left\{ (f, u_f|_{O}): \ f\in C_c^{\infty}(O), \ (f,1)_{L^2(M)}=0, \ (-\D_g)^s u_f = f \mbox{ on } M \right\} \subset \tilde{H}^{-s}(O)\times H^s(O),\\
\mathcal{C}_1 &:= \left\{ (f, v_f|_{O}): \ f\in \tilde{H}^{-1}(O), \ (f,1)_{L^2(M)}=0, \ (-\D_g) v_f = f \mbox{ on } M \right\} \subset \tilde{H}^{-1}(O)\times H^1(O),\\
\tilde{\mathcal{C}}_1 &:= \left\{ (f, \bar{v}_f|_{O}): \ f\in C_c^{\infty}(O), \ (f,1)_{L^2(M)}=0, \ \bar{v}_f = (-\D_g)^{s-1} u_f, \ (-\D_g)^s u_f = f \mbox{ on } M \right\}.
\end{align*}
As indicated in Section \ref{sec:not}, in all these definitions, we consider solutions $u_f, v_f, \bar{v}_f$ having mean zero.

We begin by discussing the relation between the sets $\mathcal{C}_s$, $\tilde{\mathcal{C}}_1$ and $\mathcal{C}_1$.

\begin{lem}
\label{lem:relations}
Let $s\in (0,1)$, let $(M,g)$ be a closed, connected, smooth manifold and let the sets $\mathcal{C}_s$, $\tilde{\mathcal{C}}_1$ and $\mathcal{C}_1$ be as above. Then,
the following results hold:
\begin{itemize}
\item[(i)] We have that
\begin{align*}
\overline{\tilde{\mathcal{C}}_1}^{\tilde{H}^{-1}(O) \times H^1(O)}   \subset \mathcal{C}_1.
\end{align*}
\item[(ii)] In the case that the metric $g \in C^{\infty}(M, \R^{n\times n}_+)$ is known, the sets $\mathcal{\tilde{C}}_1$ and $\mathcal{C}_s $ are related by a linear, bounded operator $T: \mathcal{C}_s \rightarrow \mathcal{\tilde{C}}_1 $:
\begin{align*}
T(f, u_f|_{O}) := (f,\tilde{c}_s \int\limits_{0}^{\infty} t^{1-2s} \tilde{u}_f(x,t) dt|_{O})=:(f, \bar{v}_f|_O),
\end{align*}
where $\tilde{c}_s \in \R \setminus \{0\}$, $\tilde{u}_f(x,t)$ denotes the variable coefficient Caffarelli-Silvestre-type extension extension of $u_f$ and where for $x\in M$
\begin{align*}
\bar{v}_f(x):= \tilde{c}_s \int\limits_{0}^{\infty} t^{1-2s} \tilde{u}_f(x,t) dt. 
\end{align*}
In particular, the integral on the right hand side is well-defined and finite.
Moreover, in the case that the metric $g$ is known, $u_f \in \pi_2 \mathcal{C}_s$ can be recovered from $\bar{v}_f \in \pi_2 \mathcal{\tilde{C}}_1 $ by setting 
\begin{align*}
u_f(x) = c_{1-s}\lim\limits_{y \rightarrow 0} y^{2s-1}\p_y \tilde{v}_f(x,y) = c_{\bar{s}}\lim\limits_{y \rightarrow 0} y^{1-2\bar{s}}\p_y \tilde{v}_f(x,y),
\end{align*}
where $\tilde{v}_f(x,y)$ denotes the variable coefficient Caffarelli-Silvestre-type extension of $\bar{v}_f$ (with fractional exponent $\bar{s} = 1-s \in (0,1)$) and where $\pi_2 : \R^2 \rightarrow \R $, $(x,y)\mapsto y$, denotes the projection onto the second component.
\item[(iii)] The set $\mathcal{\tilde{C}}_1$ can be constructively obtained from the set $\mathcal{C}_s$ also in the setting of the inverse problem in which only $g|_{O}$ and the data $(f,u_f|_{O}) \in \mathcal{C}_s$ with $f\in C_c^{\infty}(O)$ are given. More precisely, any element $\bar{v}_f|_{O}$ of $\pi_2\tilde{\mathcal{C}}_1$ can be reconstructed from the knowledge of an associated pair $(f,u_f|_{O}) \in \mathcal{C}_s$.
\end{itemize}
\end{lem}

\begin{rmk}
Let us comment on the conditions (ii), (iii). Indeed, in the case that $(M,g)$ is known, then the operator $T$ can be regarded as a linear, bounded operator
\begin{align*}
\tilde{T}: \{f\in \tilde{H}^{-s}(O): \ (f,1)_{L^2(M)}=0\} \rightarrow \tilde{H}^{-1}(O) \times H^1(O), \ f \mapsto (f, \bar{v}_f|_O),
\end{align*}
i.e., it can be constructed solely from $f$. The continuity then follows from the fact that for $f\in \tilde{H}^{-s}(O)$ with $(f,1)_{L^2(M)}=0$ also $(u_f,1)_{L^2(M)}=0$ (for $u_f$ defined by $(-\D_g)u_f =f $ in $M$) and that then
\begin{align*}
\|(-\D_g)^{s-1} u_f \|_{H^1(M)} \leq C\|f\|_{H^{-s}(M)}.
\end{align*}
In the setting of (iii) which is that of the inverse problem at hand, the operator $T$ acts on both components $(f,u_f|_O)$ in order to recover $\bar{v}_f$ since by the only local knowledge of $g$ no global inversion can be carried out. Instead, we use a unique continuation argument. Moreover, in this case, the operator $T$ also loses its boundedness and becomes severely unstable. In this case, the inversion of the operator is not immediate, as the variable coefficient Caffarelli-Silvestre-type extension of $\bar{v}_f$ is not immediately accessible.
\end{rmk}

Using the observations from Lemma \ref{lem:relations} in what follows, we will show that uniqueness results for the local problem imply corresponding uniqueness results for the nonlocal problem (see Theorem \ref{thm:application} below).

\begin{proof}[Proof of Lemma \ref{lem:relations}]
We first prove the claim of (i):
We note that by construction, we have that
\begin{align*}
\overline{\tilde{\mathcal{C}}_1}^{\tilde{H}^{-1}(O) \times H^1(O)}   \subset \mathcal{C}_1.
\end{align*}
Indeed, this follows from the definition of $\bar{v}_f$,
\begin{align*}
(-\D_g) \bar{v}_f =  (-\D_g) (-\D_g)^{s-1} u_f = (-\D_g)^s u_f = f \mbox{ in } M,
\end{align*}
and the continuity of the map $\tilde{H}^{-1}(O)\ni f\mapsto v_f \in H^1(O)$, where $(-\D_g) v_f = f$ in $M$.

Next, we discuss (ii). By linearity, boundedness and density, it suffices to consider $f\in C_c^{\infty}(O)$. We claim that for some constant $\tilde{c}_s \neq 0$
\begin{align*}
\tilde{c}_s  \int\limits_{0}^{\infty} t^{1-2s} \tilde{u}_f(x,t) dt = (-\D_g)^{s-1} u_f(x), \ \mbox{ for } x\in M.
\end{align*}
In particular, this then proves the well-definedness of the integral.
The claimed identity follows, for instance, from the series expansion into an $L^2$-normalized eigenbasis $\{\phi_k\}_{k\in \N}$ of the Laplace-Beltrami operator $-\D_g$ on $(M,g)$ with associated eigenvalues $\{\lambda_k\}_{k\in \N}$ as in Appendix \ref{app:exp}. Indeed, for $x\in M$ and $f\in C_c^{\infty}(O)$ with $(f,1)_{L^2(M)}=0$, by definition of $\tilde{u}_f$, by the kernel representations from Lemma \ref{lem:energy_est}, we obtain
\begin{align*}
\int\limits_{0}^{\infty} t^{1-2s} \tilde{u}_f(x,t) dt
&= \bar{c}_{s,N}\sum\limits_{k=1}^{\infty} (f,\phi_k)_{L^2(M)} \int\limits_{0}^{\infty} t^{1-2s} K_s(t \sqrt{\lambda_k}) \left( \frac{t}{\sqrt{\lambda_k}} \right)^{s} dt \phi_k(x)\\
& = \bar{c}_{s,N}\left(\int\limits_{0}^{\infty} z^{1-s} K_s(z) dz \right) \sum\limits_{k=1}^{\infty} \lambda_k^{s-1} \lambda_k^{-s} (\phi_k, f)_{L^2(M)} \phi_k(x) \\
& = \bar{c}_{s,N}\left(\int\limits_{0}^{\infty} z^{1-s} K_s(z) dz \right) \sum\limits_{k=1}^{\infty} \lambda_k^{s-1}  (\phi_k, u_f)_{L^2(M)} \phi_k(x) \\
&= \bar{c}_{s,N}\left(\int\limits_{0}^{\infty} z^{1-s} K_s(z) dz \right) (-\D_g)^{s-1} u_f(x).
\end{align*}
Here, we first expanded $\tilde{u}_f$ as in Lemma \ref{lem:energy_est}, then changed variables, invoked the relation $(-\D_g)^s u_f =f $ in $M$ and finally invoked the spectral definition of the fractional Laplacian.
As in Appendix \ref{app:exp}, we have used the notation $K_s$ to denote a modified Bessel function of the second kind. By its asymptotics (see, for instance, \cite{NIST}) it follows that for $s\in (0,1)$ the integrals $\int\limits_{0}^{\infty} z^{1-s} K_s(z) dz $ are well-defined and finite.

In case that the metric is known on the whole of $M$, the recovery of $u_f$ from $\bar{v}_f$ follows from duality. Indeed, we have that weakly 
\begin{align*}
c_{1-s}\lim\limits_{y \rightarrow 0} y^{2s-1}\p_y \tilde{v}_f(x,y)
= (-\D_g)^{1-s} \bar{v}_f(x) = (-\D_g)^{1-s} (-\D_g)^{s-1} u_f(x) = u_f(x), \ x \in M.
\end{align*}
This concludes the proof of (ii).\\

Last but not least, we turn to the claim of (iii),
stating that $\bar{v}_f|_{O}$ can be constructed from the knowledge of the pair $(f,u_f|_{O})$ and $g|_O$, i.e., that $\tilde{\mathcal{C}}_1$ is completely (and constructively) determined by the set $\mathcal{C}_s$. One way to see this is to firstly recall from (ii) above that
\begin{align*}
\bar{v}_f(x) = \int\limits_{0}^{\infty} t^{1-2s} \tilde{u}_f(x,t) dt,
\end{align*}
where $\tilde{u}_f$ is the Caffarelli-Silvestre-type extension of $u_f$. Hence, it suffices to reconstruct $\tilde{u}_f|_{O \times \R_+}$ from the associated pair $(f, u_f|_O) \in \mathcal{C}_s$.
Secondly, since the metric $g$ is assumed to be known in $O$ and since $\tilde{u}_f(x,0)|_{O}$ and $\lim\limits_{x_{n+1}\rightarrow 0} x_{n+1}^{1-2s} \p_{x_{n+1}} \tilde{u}_f(x,x_{n+1})|_{x\in O} = f(x)|_{x\in O}$ are known by the data contained in $\mathcal{C}_s$, the function $\tilde{u}_f|_{O \times \R_+}$ can be recovered by unique continuation \cite{FF14, R15, GSU16, GRSU20}. For instance, a Tikhonov regularization approach allows to implement this constructively.
\end{proof}

We next claim that $\tilde{\mathcal{C} }_1  \subset \mathcal{C}_1$ is dense:

\begin{thm}[Density]
\label{thm:density}
Let $s\in (0,1)$. Let $(M,g)$ be a smooth, closed, connected manifold of dimension $n\geq 2$.
Then,
\begin{align*}
\overline{\tilde{\mathcal{C}}_1}^{\tilde{H}^{-1}(O) \times H^1(O)}  = \mathcal{C}_1 \subset \tilde{H}^{-1}(O) \times H^1(O).
\end{align*}
\end{thm}

\begin{proof}
We first show that $\pi_2 \mathcal{\tilde{C}}_1$ is dense in $\pi_2 \mathcal{C}_1$ with respect to the $H^1(O)$ topology, where $\pi_2$ denotes the projection with respect to the second component. The overall density then follows from the continuity of the map $v_f \mapsto f$ as a map from $H^{1}(O)$ to $\tilde{H}^{-1}(O)$.

We consider $\psi \in \tilde{H}^{-1}(O)$ such that $\psi(\bar{v}_f) = 0$ for all $f\in C_c^{\infty}(O)$ with $(f,1)_{L^2(M)}=0$ and $\bar{v}_f \in \pi_2 \mathcal{\tilde{C}}_1$. 
By the Hahn-Banach theorem it suffices to show that then also 
\begin{align}
\label{eq:HB_claim}
\psi(v_f) = 0 \mbox{ for all } v_f \in \pi_2 \mathcal{C}_1.
\end{align}
Since, by construction, $(\bar{v}_f, 1)_{L^2(M)}=0 = (v_f, 1)_{L^2(M)}$, for all $\bar{v}_f \in \pi_2 \mathcal{\tilde{C}}_1$ and $v_f \in \pi_2 \mathcal{C}_1$, we may assume that also $(\psi,1)_{L^2(M)} =0$.

To prove \eqref{eq:HB_claim}, we define the auxiliary function $w \in H^1(M)$ to be the unique solution to $(-\D_g) w = \psi$ in $M$ with $(w,1)_{L^2(M)}=0$. 
Further, we observe that for all $f\in C_c^{\infty}(O)$ with $(f,1)_{L^2(M)}=0$ it holds that
\begin{align*}
0 = \psi(\bar{v}_f) = \langle (-\D_g)^s u_f, w \rangle_{\tilde{H}^{-1}(M), H^{1}(M)} 
= ( f, w )_{L^2(O)},
\end{align*}
where $\langle \cdot, \cdot \rangle_{\tilde{H}^{-1}(M), H^{1}(M)}  $ denotes the $\tilde{H}^{-1}(M), H^{1}(M) $ duality pairing.
Hence, $w = 0$ in $O$.
As a consequence, for $v_f \in \pi_2 \mathcal{C}_1$ we have
\begin{align*}
\psi(v_f) = \langle (-\D_g) v_f, w \rangle_{\tilde{H}^{-1}(M),H^{1}(M)} = \langle f, w \rangle_{\tilde{H}^{-1}(M),H^{1}(M)} = 0.
\end{align*}
Here we have used the equation $(-\D_g) v_f = f$ on $M$ with $\supp(f) \subset O$ and the fact that $w|_{O}=0$.
This proves the desired density of $\pi_2 \mathcal{\tilde{C}}_1$  in $\pi_2 \mathcal{C}_1$. The full density result then follows as explained above.
\end{proof}

We remark that the combination of Lemma \ref{lem:relations} and Theorem \ref{thm:density} also prove the statement of Theorem \ref{thm:density1}.

As a consequence, as in \cite{CGRU23} any local uniqueness result translates into a nonlocal one. For completeness, we formulate a sample result for this. 

\begin{thm}
\label{thm:application}
Let $s\in (0,1)$.
Let $(M_j,g_j)$, $j\in \{1,2\}$, be two smooth, closed, connected manifolds of dimension $n\geq 2$. Let $(O_1,g_1)\subset (M_1,g_1)$ and $(O_2,g_2)\subset (M_2,g_2)$ be open sets such that $(O_1,g_1)=(O_2,g_2)=:(O,g)$.
Let $L_{1,O}^{g_j}$, $j\in \{1,2\}$, denote the source-to-solution maps 
\begin{align*}
\tilde{H}^{-1}(O) \ni f \mapsto v_f^{j}|_{O} \in H^1(O), 
\end{align*}
where $v_f^j $ is a solution to $(-\D_g) v_f^j = f$ in $(M,g_j)$. Assume that there is uniqueness (up to the natural gauge) in the local problem with source-to-solution data, i.e., that the following holds:
If $L_{1,O}^{g_1} = L_{1,O}^{g_2}$, then there exists a smooth diffeomorphism $\Phi$ such that $\Phi^{\ast}g_2 =g_1$.

Then, there is also uniqueness (up to the natural gauge) in the nonlocal problem with source-to-solution data, i.e., if $L_{s,O}^{g_1} = L_{s,O}^{g_2}$, then there exists a smooth diffeomorphism $\Phi$ such that $\Phi^{\ast}g_2 =g_1$.
\end{thm}

Although this follows along the same lines as in \cite{GU21} and \cite{CGRU23} we give the short proof for completeness.

\begin{proof}
If $L_{s,O}^{g_1} = L_{s,O}^{g_2}$, then, by Lemma \ref{lem:relations} and Theorem \ref{thm:density} it also holds that $L_{1,O}^{g_1} = L_{1,O}^{g_2}$. Now by the assumed uniqueness result for the local operator, as desired, this implies that the metrics $g_1, g_2$ are related by a smooth diffeomorphism $\Phi$ such that $\Phi^{\ast}g_2 =g_1$.
\end{proof}

\section*{Acknowledgements}

The author gratefully acknowledges support by the Deutsche Forschungsgemeinschaft (DFG, German Research Foundation) through the Hausdorff Center for Mathematics under Germany’s Excellence Strategy - EXC-2047/1 - 390685813.

\appendix
\section{Expansion into Eigenfunctions}
\label{app:exp}

For completeness and self-containedness, we present several eigenfunction expansions for the solutions to the Dirichlet and Neumann Poisson problems \eqref{eq:Dir}, \eqref{eq:Neu}. We emphasize that these arguments are well-known and can be found in the diagonalization ideas in \cite[Chapter 3]{ST10} or \cite{CS07}, for instance.

\begin{lem}[Eigenfunction expansion of the Neumann and Poisson kernels]
\label{lem:expansion}
Let $s\in (0,1)$ and let $(M,g)$ be a smooth, closed, connected manifold. Let $P_{y}(x,y)$ and $P_{y}^D(x,y)$ denote the Neumann and Dirichlet Poisson kernels defined in \eqref{eq:Poisson} and \eqref{eq:PoissonD}. Let $\{\phi_k\}_{k\in \N}$, $\{\lambda_k\}_{k\in \N}$ with $0=\lambda_0<\lambda_1 \leq \lambda_k$, $k\geq 1$, denote an $L^2$-normalized basis of eigenfunctions and the sequence of associated eigenvalues for the Laplace-Beltrami operator on $(M,g)$. 
Then we have that for some constant $\bar{c}_{s,N}, \bar{c}_{s,D}\neq 0$
\begin{align*}
\int\limits_{M} P_{y}(x,z) f(z) dz& = \bar{c}_{s,N}\sum\limits_{k=1}^{\infty} (f,\phi_k)_{L^2(M)} K_{s}(\sqrt{\lambda_k} y) \left( \frac{y}{\sqrt{\lambda_k}} \right)^{s} \phi_k(x),\\
& \qquad \qquad \ \mbox{ where } f\in L^2(M) \mbox{ with }
(f,1)_{L^2(M)}=0,\\
\int\limits_{M} P_{y}^D(x,z) f(z) dz& =  \bar{c}_{s,D} \sum\limits_{k=0}^{\infty} (f,\phi_k)_{L^2(M)} K_{s}(\sqrt{\lambda_k} y) \left( y \sqrt{\lambda_k}\right)^{s} \phi_k(x),\\
& \qquad \qquad \ \mbox{ where } f\in L^2(M).
\end{align*}
Here $K_s$ denotes the modified Bessel function of the second kind.
\end{lem}

We only present this for completeness, similar arguments are available in the literature, see Remark \ref{rmk:ODE}.

\begin{proof}
We use the definitions from \eqref{eq:Poisson}, \eqref{eq:PoissonD} and only consider $f\in C^{\infty}(M)$ with $(f,1)_{L^2(M)}=0$ by density.
As the arguments are similar, we only discuss the case of the Neumann Poisson kernel in more detail. To this end, we note that the heat kernel can be expressed in terms of an eigenfunction expansion: For $t\geq 0$, $x,z \in M$
\begin{align*}
K_t(x,z) = \sum\limits_{k=0}^{\infty} e^{-\lambda_k t} \phi_k(x) \phi_k(z).
\end{align*}
Hence, for $f\in C^{\infty}(M)$ with $(f,1)_{L^2(M)}=0$ and for $x\in M$, $y\geq 0$,
\begin{align*}
\int\limits_{M} P_{y}(x,z) f(z) dz&
= \tilde{c}_s \sum\limits_{k=1}^{\infty} (\phi_k, f)_{L^2(M)} \int\limits_{0}^{\infty} e^{-\lambda_k t} e^{ - \frac{y^2}{4t}} \frac{dt}{t^{1-s}}.
\end{align*}
Changing variables by setting $r = \lambda_k t$ and using that by formula \cite[10.32.10]{NIST} we have
\begin{align*}
K_s(z) = \frac{1}{2} \left( \frac{1}{2} z \right)^{-s} \int\limits_{0}^{\infty} \exp\left(-t- \frac{z^2}{4t} \right) \frac{dt}{t^{1-s}},
\end{align*}
we conclude that
\begin{align*}
\int\limits_{M} P_{y}(x,z) f(z) dz
&=
\tilde{c}_s \sum\limits_{k=1}^{\infty} (\phi_k, f)_{L^2(M)} \int\limits_{0}^{\infty} e^{-\lambda_k t} e^{ - \frac{y^2}{4t}} \frac{dt}{t^{1-s}}\\
&= 2 \tilde{c}_s \sum\limits_{k=1}^{\infty} (\phi_k, f)_{L^2(M)} \lambda_k^{-s} \left( \frac{\sqrt{\lambda_k} y}{2} \right)^s K_s(\sqrt{\lambda}_k y).
\end{align*}
Comparing this with the claimed expression and recalling that $f\in C^{\infty}(M)$ with $(f,1)_{L^2(M)}=0$ was arbitrary, hence yield the claimed identity.
\end{proof}

\begin{rmk}
\label{rmk:ODE}
As an alternative to the proof presented above, one could also argue by using that the functions on the left hand side are decaying solutions (in $y$) of the corresponding Caffarelli-Silvestre-type extensions \eqref{eq:Dir}, \eqref{eq:Neu}. With this one can argue by an ODE reduction similarly as in \cite{ST10, GFR19}.
\end{rmk}

With these representations in hand, we deduce energy estimates for solutions $\tilde{u}$ of the Neumann and Dirichlet problems \eqref{eq:Neu}, \eqref{eq:Dir}, respectively.

\begin{lem}[Energy estimates]
\label{lem:energy_est}
Let $s\in (0,1)$ and let $(M,g)$ be a smooth, closed, connected manifold. Let $P_{y}(x,z)$ and $P_{y}^D(x,z)$ denote the Neumann and Dirichlet Poisson kernels defined in \eqref{eq:Poisson} and \eqref{eq:PoissonD} for $x,z\in M$ and $y\geq 0$. 
Let for $(x,x_{n+1}) \in M \times \R_+$
\begin{align*}
\tilde{u}(x,x_{n+1})&:= \int\limits_{M} P_{x_{n+1}}(x,z) f(z) dz, \mbox{ for } f\in L^2(M) \mbox{ with } (f,1)_{L^2(M)}=1,\\
\tilde{v}(x,x_{n+1})&:= \int\limits_{M} P_{x_{n+1}}^D(x,z) f(z) dz, \mbox{ for } f\in L^2(M) .
\end{align*}
Then, $\tilde{u}, \tilde{v} \in \dot{H}^{1}(M \times \R_+, x_{n+1}^{1-2s})$. In particular, the thus defined functions are the unique solutions of the Caffarelli-Silvestre-type extension problems \eqref{eq:Neu}, \eqref{eq:Dir} in their energy class (and with the orthogonality condition formulated for $f$).
\end{lem}

\begin{proof}
We discuss the regularity properties of $\tilde{u}$; the argument for $\tilde{v}$ is similar. By Lemma \ref{lem:expansion} and since $(f,1)_{L^2(M)}=0$, for the tangential estimates we have
\begin{align*}
\int\limits_{0}^{\infty} \int\limits_{M} y^{1-2s} |\nabla_x \tilde{u}(x,y)|^2 dx dy
= \sum\limits_{k=1}^{\infty} |(f,\phi_k)_{L^2(M)}|^2 \lambda_k \int\limits_{0}^{\infty} \left|K_s(\sqrt{\lambda_k} y) \left( \frac{y}{\sqrt{\lambda_k}} \right)^{s} \right|^2 y^{1-2s} dy.
\end{align*}
Changing coordinates by setting $z= \sqrt{\lambda_k} y$ then yields that
\begin{align*}
 &\sum\limits_{k=1}^{\infty} |(f,\phi_k)_{L^2(M)}|^2 \lambda_k \int\limits_{0}^{\infty} \left|K_s(\sqrt{\lambda_k} y) \left( \frac{y}{\sqrt{\lambda_k}} \right)^{s} \right|^2 y^{1-2s} dy\\
 & =  \left( \int\limits_{0}^{\infty} |K_s(z)|^2|z| dz \right) \sum\limits_{k=1}^{\infty} |(f,\phi_k)_{L^2(M)}|^2 \lambda_k^{-s}.
\end{align*}
By virtue of the asymptotics of the modified Bessel functions $\left( \int\limits_{0}^{\infty} |K_s(z)|^2|z| dz \right) = c(s)<\infty$ and since $\lambda_k \rightarrow \infty$ for $k \rightarrow \infty$ the claimed tangential integrability thus follows.

Concerning the normal derivative, we record that by \cite[Formula 10.29.2]{NIST} 
\begin{align*}
K_s'(z) = K_{1-s}(z) - \frac{s}{z}K_s(z).
\end{align*}
Hence, 
\begin{align*}
\int\limits_{0}^{\infty} \int\limits_{M} y^{1-2s} |\p_y \tilde{u}(x,y)|^2 dx dy
= \sum\limits_{k=1}^{\infty} |(f,\phi_k)_{L^2(M)}|^2  \lambda_k  \int\limits_{0}^{\infty} \left|K_{1-s}(\sqrt{\lambda_k} y) \left( \frac{y}{\sqrt{\lambda_k}} \right)^{s} \right|^2 y^{1-2s} dy.
\end{align*}
Changing coordinates  $z= \sqrt{\lambda_k} y$ as above, we infer that 
\begin{align*}
&\sum\limits_{k=1}^{\infty} |(f,\phi_k)_{L^2(M)}|^2  \lambda_k  \int\limits_{0}^{\infty} \left|K_{1-s}(\sqrt{\lambda_k} y) \left( \frac{y}{\sqrt{\lambda_k}} \right)^{s} \right|^2 y^{1-2s} dy\\
 & =  \left( \int\limits_{0}^{\infty} |K_{1-s}(z)|^2|z|  dz \right) \sum\limits_{k=1}^{\infty} |(f,\phi_k)_{L^2(M)}|^2 \lambda_k^{-s}.
\end{align*}
Analogously as above, the asymptotics of the modified Bessel functions imply that 
\begin{align*}
\left( \int\limits_{0}^{\infty} |K_{1-s}(z)|^2|z| dz \right) = c(s)<\infty. 
\end{align*}
Hence, we have also obtained the desired normal integrability.

The argument for the Dirichlet case is similar.
\end{proof}

\bibliographystyle{alpha}
\bibliography{citations_v8}

\end{document}